\documentclass[review,english]{elsarticle_mod}

\usepackage{lineno,hyperref}
\modulolinenumbers[5]
\PassOptionsToPackage{normalem}{ulem}
\usepackage{ulem}
\usepackage{babel}
\usepackage{color}
\usepackage{amsmath}
\usepackage{amssymb}
\usepackage{amsthm}
\usepackage{geometry}
\makeatletter
\numberwithin{equation}{section}
\numberwithin{figure}{section}
  \theoremstyle{remark}
  \newtheorem*{notation*}{\protect\notationname}
\theoremstyle{plain}
\newtheorem{thm}{\protect\theoremname}
  \theoremstyle{definition}
  \newtheorem{defn}[thm]{\protect\definitionname}
  \theoremstyle{plain}
  \newtheorem{lem}[thm]{\protect\lemmaname}
  \theoremstyle{remark}
  \newtheorem{rem}[thm]{\protect\remarkname}

\makeatother

  \providecommand{\definitionname}{Definition}
  \providecommand{\lemmaname}{Lemma}
  \providecommand{\notationname}{Notation}
  \providecommand{\remarkname}{Remark}
\providecommand{\theoremname}{Theorem}

\begin{document}

\begin{frontmatter}

\title{Asymptotic analysis of a semi-linear elliptic system in perforated
domains: well-posedness and correctors for the homogenization limit}


\author[aff1]{Vo Anh Khoa\corref{mycorrespondingauthor}}
\cortext[mycorrespondingauthor]{Corresponding author}
\ead{khoa.vo@gssi.infn.it, vakhoa.hcmus@gmail.com}

\author[aff2]{Adrian Muntean}
\ead{adrian.muntean@kau.se}

\address[aff1]{Mathematics and Computer Science Division,
Gran Sasso Science Institute, L'Aquila, Italy}
\address[aff2]{Department of Mathematics and Computer Science, Karlstad University,
Sweden}

\begin{abstract}
In this study, we prove results on the weak solvability and homogenization
of a microscopic semi-linear elliptic system posed in
perforated media. The model presented here explores 
the interplay between stationary diffusion and both surface and volume
chemical reactions in porous media. Our interest lies in deriving homogenization limits (upscaling) for alike systems and particularly in justifying rigorously the obtained averaged descriptions. Essentially, we prove the well-posedness
of the microscopic problem ensuring also the positivity and boundedness
of the involved concentrations and then use the structure of the two scale expansions to derive corrector estimates delimitating this way the convergence rate of the asymptotic approximates to the macroscopic limit concentrations. Our techniques include Moser-like iteration techniques, a variational formulation, two-scale asymptotic expansions as well as energy-like estimates.  
\end{abstract}


\end{frontmatter}

\linenumbers

\section{Introduction}

We study the semi-linear elliptic boundary-value
problem of the form
\begin{equation}
\left(P^{\varepsilon}\right):\;\begin{cases}
\mathcal{A}^{\varepsilon}u_{i}^{\varepsilon}\equiv\nabla\cdot\left(-d_{i}^{\varepsilon}\nabla u_{i}^{\varepsilon}\right)=R_{i}\left(u^{\varepsilon}\right), & \mbox{in}\;\Omega^{\varepsilon}\subset\mathbb{R}^{d},\\
d_{i}^{\varepsilon}\nabla u_{i}^{\varepsilon}\cdot\mbox{n}=\varepsilon\left(a_{i}^{\varepsilon}u_{i}^{\varepsilon}-b_{i}^{\varepsilon}F_{i}\left(u_{i}^{\varepsilon}\right)\right), & \mbox{on}\;\Gamma^{\varepsilon},\\
u_{i}^{\varepsilon}=0, & \mbox{on}\;\Gamma^{ext},
\end{cases}\label{Pep}
\end{equation}
for $i\in\left\{ 1,...,N\right\} $ ($N\ge2, d\in\{2,3\}$). Following \cite{AWR}, this system models the diffusion in a porous medium as well as the aggregation, dissociation and surface deposition of $N$ interacting populations of colloidal particles  indexed by 
 $u_{i}^{\varepsilon}$.  As short-hand notation,  $u^{\varepsilon}:=\left(u_{1}^{\varepsilon},...,u_{N}^{\varepsilon}\right)$
points out the vector of these concentrations.  Such scenarios arise in drug-delivery mechanisms in human bodies and often includes cross- and thermo-diffusion which are triggers of our motivation (compare \cite{dGM62} for the Sorret and Dufour effects and \cite{FIMU12,VE09} for related cross-diffusion
and chemotaxis-like systems).

The model \eqref{Pep} involves a
number of parameters: $d_{i}^{\varepsilon}$ represents molecular
diffusion coefficients, $R_{i}$ represents the volume reaction rate, 
$a_{i}^{\varepsilon},b_{i}^{\varepsilon}$ are the so-called deposition coefficients, while $F_{i}$ indicates a
surface chemical reaction for the immobile species. We refer to \eqref{Pep} as problem  $\left(P^{\varepsilon}\right)$.

The main purpose of this paper is to obtain corrector estimates that delimitate the error made when homogenizing (averaging, upscaling, coarse graining...) the problem  $\left(P^{\varepsilon}\right)$, i.e. we want to estimate the speed of convergence as $\varepsilon \to 0$ of suitable norms of differences in micro-macro concentrations and micro-macro concentration gradients.  This way we justify rigorously the upscaled models derived in \cite{AWR} and prepare the playground to obtain corrector estimates for the thermo-diffusion scenario discussed in \cite{NHM}. 
From the corrector estimates perspective, the major mathematical difficulty we meet here is the presence of the nonlinear surface reaction term. To quantify its contribution to the corrector terms we use an energy-like approach very much inspired by \cite{CP99}. The main result of the paper is Theorem \ref{thm:10} where we state the corrector estimate. It is worth noting that this work goes along the line open by our works \cite{Ptash} (correctors via periodic unfolding) and \cite{Tycho} (correctors by special test functions adapted to the local periodicity of the microstructures). An alternative strategy to derive correctors for our scenario  could in principle exclusively rely on periodic unfolding, refolding and defect operators approach  if the boundary conditions along the microstructure would be of homogeneous Neumann type; compare \cite{Sina1} and \cite{Sina2}. 

The corrector estimates obtained with this framework can
be further used to design convergent multiscale finite element methods for
the studied PDE system (see e.g. \cite{Hou1999} for the basic idea of
the  MsFEM approach and \cite{Legoll2014} for an application to perforated
media).

The paper is organized as follows: In Section \ref{TP} we start off with a  set of technical preliminaries focusing especially on the working assumptions on the data and the description of the microstructure  of the porous medium.   The weak solvability of the microscopic model  is established in Section \ref{WP}. The homogenization method is applied in Section \ref{HM} to the problem $\left(P^{\varepsilon}\right)$. This is the place where we derive the corrector estimates and establish herewith the convergence rate of the homogenization process. A brief discussion (compare Section \ref{D}) closes the paper. 

\section{Preliminaries}\label{TP}

\subsection{Description of the geometry}

The geometry of our porous medium is sketched in Figure \ref{fig:1} (left), together with the choice of perforation (referred here to also as "microstructure") cf. Figure \ref{fig:1} (right).  We refer the reader to 
\cite{HJ91} for a concise mathematical representation of the perforated geometry. 
In the same spirit, take  $\Omega$ be a bounded open domain in $\mathbb{R}^{d}$ with a
piecewise smooth boundary $\Gamma=\partial\Omega$. Let $Y$ be the
unit representative cell, i.e. 
\[
Y:=\left\{ \sum_{i=1}^{d}\lambda_{i}\vec{e}_{i}:0<\lambda_{i}<1\right\} ,
\]
where we denote by $\vec{e}_{i}$ by $i$th unit vector in $\mathbb{R}^{d}$.

Take $Y_{0}$ the open subset of $Y$ with a piecewise
smooth boundary $\partial Y_{0}$ in such a way that $\overline{Y_{0}}\subset\overline{Y}$.
In the porous media terminology, $Y$ is the unit cell made of two
parts: the gas phase (pore space) $Y\backslash\overline{Y_{0}}$ and
the solid phase $Y_{0}$.

Let $Z\subset\mathbb{R}^{d}$ be a hypercube. Then for $X\subset Z$
we denote by $X^{k}$ the shifted subset
\[
X^{k}:=X+\sum_{i=1}^{d}k_{i}\vec{e}_{i},
\]
where $k=\left(k_{1},...,k_{d}\right)\in\mathbb{Z}^{d}$ is a vector
of indices.

Setting $Y_{1}=Y\backslash\overline{Y_{0}}$, we now define the pore
skeleton by
\[
\Omega_{0}^{\varepsilon}:=\bigcup_{k\in\mathbb{Z}^{d}}\left\{ \varepsilon Y_{0}^{k}:Y_{0}^{k}\subset\Omega\right\} ,
\]
where $\varepsilon$ is observed as a given scale factor or homogenization
parameter.

It thus comes out that the total pore space is
\[
\Omega^{\varepsilon}:=\Omega\backslash\overline{\Omega_{0}^{\varepsilon}},
\]
for $\varepsilon Y_{0}^{k}$ the $\varepsilon$-homotetic set of $Y_{0}^{k}$,
while the total pore surface of the skeleton is denoted by
\[
\Gamma^{\varepsilon}:=\partial\Omega_{0}^{\varepsilon}=\bigcup_{k\in\mathbb{Z}^{d}}\left\{ \varepsilon\Gamma^{k}:\Gamma^{k}\subset\Omega\right\} .
\]
The exterior boundary of $\Omega^{\varepsilon}$
is certainly a hypersurface in $\mathbb{R}^{d}$, denoted by $\Gamma^{ext}=\partial\Omega^{\varepsilon}\backslash\Gamma^{\varepsilon}$, 
where it has a nonzero $\left(d-1\right)$-dimensional measure, satisfies
$\Gamma^{ext}\cap\Gamma^{\varepsilon}=\emptyset$ and coincides with
$\Gamma$. Moreover, $\mbox{n}$ denotes the unit normal vector to
$\Gamma^{\varepsilon}$.

Finally, our perforated domain $\Omega^{\varepsilon}$ is assumed
to be connected through the gas phase. Notice here that $\Gamma^{ext}$
is smooth.

\begin{figure}[!h]
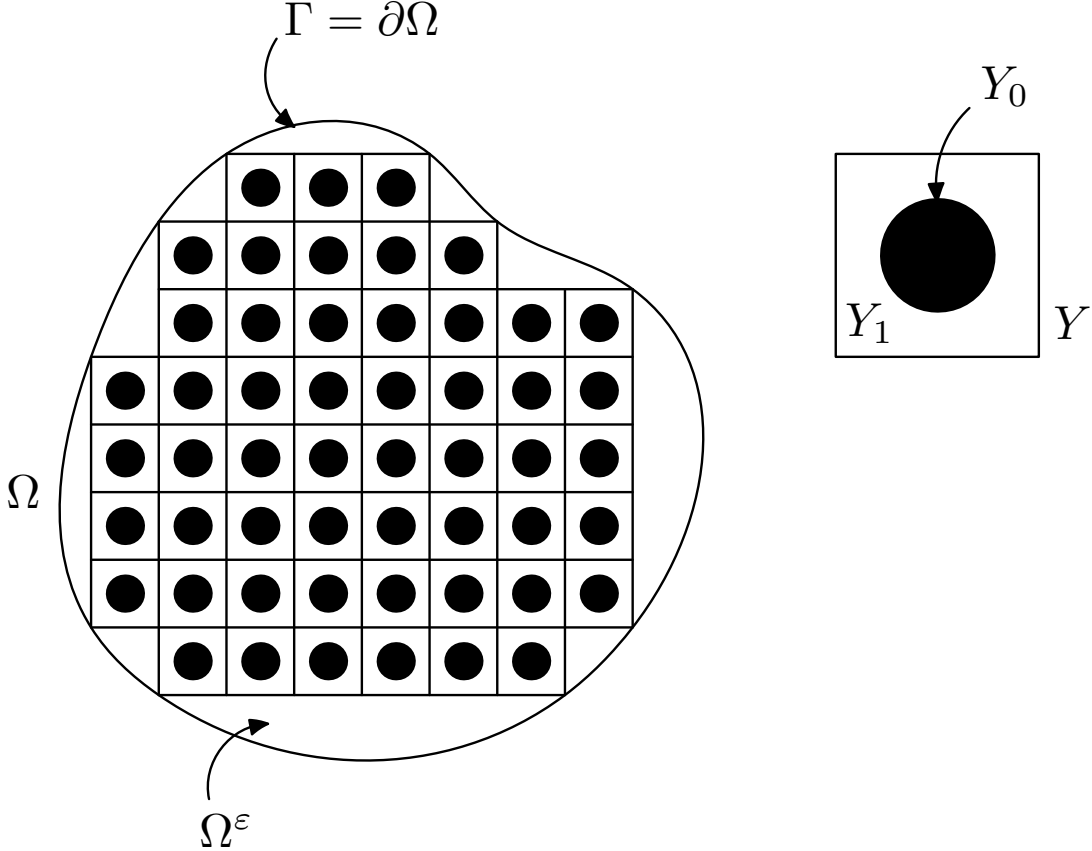
 	
\centering 	
	\parbox{10cm}{\convertMPtoPDF{fig.man}{1.8}{1.8}}	
	\caption{Admissible two-dimensional perforated domain (left) and basic geometry of the microstructure (right).}
	\label{fig:1}
\end{figure}

N.B. This paper aims at understanding the problem in two or three space dimensions. However, all our results hold also for $d\geq 3$. Throughout this paper, $C$ denotes a generic constant which can change
from line to line. If not otherwise stated, the constant $C$ is independent
of the choice of $\varepsilon$. 

\subsection{Notation. Assumptions on the data}

We denote by  $x\in\Omega^{\varepsilon}$ the macroscopic variable
 and by $y=x/\varepsilon$ the microscopic variable representing  fast variations at the microscopic
geometry. With this convention in view, we write
\[
d_{i}^{\varepsilon}\left(x\right)=d_{i}\left(\frac{x}{\varepsilon}\right)=d_{i}\left(y\right).
\]

A similar meaning is given to all involved "oscillating" data, e.g. to $a_{i}^{\varepsilon}\left(x\right)$,
$b_{i}^{\varepsilon}\left(x\right)$.

We now make the following set of assumptions:

$\left(\mbox{A}_{1}\right)$ the diffusion coefficient $d_{i}^{\varepsilon}\in L^{\infty}\left(\mathbb{R}^{d}\right)$
is $Y$-periodic, and it exists a positive constant $\alpha_{i}$
such that
\[
d_{i}\left(y\right)\xi_{i}\xi_{j}\ge\alpha_{i}\left|\xi\right|^{2}\quad\mbox{for any}\;\xi\in\mathbb{R}^{d}.
\]

$\left(\mbox{A}_{2}\right)$ the deposition coefficients $a_{i}^{\varepsilon},b_{i}^{\varepsilon}\in L^{\infty}\left(\Gamma^{\varepsilon}\right)$
are positive and $Y$-periodic.

$\left(\mbox{A}_{3}\right)$ the reaction rates $R_{i}:\Omega^{\varepsilon}\times\left[0,\infty\right)^{N}\to\mathbb{R}$
and $F_{i}:\Gamma^{\varepsilon}\times\left[0,\infty\right)\to\mathbb{R}$
are Carath\'eodory functions, i.e. they are, respectively, continuous
in $\left[0,\infty\right)^{N}$ and $\left[0,\infty\right)$ with
respect to $x$ variable (in the ``almost all'' sense), and measurable
in $\Omega^{\varepsilon}$ and $\Gamma^{\varepsilon}$ with essential
boundedness with respect to concentrations $u_{i}^{\varepsilon}\ge0$.

$\left(\mbox{A}_{4}\right)$ The chemical rate $R_i$ and $F_i$ are sublinear
in the sense that for any  $p=\left(p_{1},...,p_{N}\right)$
\[
R_{i}\left(p\right)\le C\left(1+\sum_{j=1,j\ne i}^{N}p_{i}p_{j}\right)\quad\mbox{for}\;p\ge0,
\]
\[
F_{i}\left(p_{i}\right)\le C\left(1+p_{i}\right)\quad\mbox{for}\;p_{i}\ge0,\]
for any  $p=\left(p_{1},...,p_{N}\right)$.

Furthermore, assume that $R_{i}\left(p\right)/p_{i}$ is decreasing
and $F_{i}\left(p_{i}\right)/p_{i}$ is increasing in $p_{i}$ for any 
$p>0$.

$\left(\mbox{A}_{5}\right)$ For every $\varepsilon>0$, there exist
vectors ($x$-dependent) $r_{0}^{\varepsilon},r_{\infty}^{\varepsilon},f_{0}^{\varepsilon},f_{\infty}^{\varepsilon}$
whose elements are
\[
r_{0,i}^{\varepsilon}=\lim_{u_{i}^{\varepsilon}\to0^{+}}\frac{R_{i}\left(u^{\varepsilon}\right)}{u_{i}^{\varepsilon}},\quad r_{\infty,i}^{\varepsilon}=\lim_{u_{i}^{\varepsilon}\to\infty}\frac{R_{i}\left(u^{\varepsilon}\right)}{u_{i}^{\varepsilon}},
\]
\[
f_{0,i}^{\varepsilon}=\lim_{u_{i}^{\varepsilon}\to0^{+}}\varepsilon\left(a_{i}^{\varepsilon}-b_{i}^{\varepsilon}\frac{F_{i}\left(u_{i}^{\varepsilon}\right)}{u_{i}^{\varepsilon}}\right),\quad f_{\infty,i}^{\varepsilon}=\lim_{u_{i}^{\varepsilon}\to\infty}\varepsilon\left(a_{i}^{\varepsilon}-b_{i}^{\varepsilon}\frac{F_{i}\left(u_{i}^{\varepsilon}\right)}{u_{i}^{\varepsilon}}\right).
\]

$\left(\mbox{A}_{6}\right)$ $R_{i}$ and $F_{i}$ satisfy the growth conditions:  
\begin{equation}
\left|R_{i}\left(x,p\right)\right|\le C\sum_{i=1}^{N}\left(1+p_{i}\right)^{2}\quad\mbox{for}\;p\ge0,\label{eq:c1}
\end{equation}
\begin{equation}
\left|a_{i}^{\varepsilon}p_{i}-b_{i}^{\varepsilon}F_{i}\left(p_{i}\right)\right|\le C\left(1+p_{i}\right)^{2}\quad\mbox{for}\;p_{i}\ge0.\label{eq:c2}
\end{equation}

Let us define the function space
\[
V^{\varepsilon}:=\left\{ v\in H^{1}\left(\Omega^{\varepsilon}\right)|v=0\;\mbox{on}\;\Gamma^{ext}\right\} ,
\]
which is a closed subspace of the Hilbert space $H^{1}\left(\Omega^{\varepsilon}\right)$,
and thus endowed with the semi-norm
\[
\left\Vert v\right\Vert _{V^{\varepsilon}}=\left(\sum_{i=1}^{d}\int_{\Omega^{\varepsilon}}\left|\frac{\partial v}{\partial x_{i}}\right|^{2}dx\right)^{1/2}\quad\mbox{for all}\;v\in V^{\varepsilon}.
\]

Obviously, this norm is equivalent to the usual $H^{1}$-norm by the
Poincar\'e inequality. Moreover, this equivalence is uniform in $\varepsilon$
(cf. \cite[Lemma 2.1]{CP99}).

We introduce the Hilbert spaces
\[
\mathcal{H}\left(\Omega^{\varepsilon}\right)=L^{2}\left(\Omega^{\varepsilon}\right)\times...\times L^{2}\left(\Omega^{\varepsilon}\right),\quad\mathcal{V}^{\varepsilon}=V^{\varepsilon}\times...\times V^{\varepsilon},
\]
with the inner products defined  respectively by
\[
\left\langle u,v\right\rangle _{\mathcal{H}\left(\Omega^{\varepsilon}\right)}:=\sum_{i=1}^{N}\int_{\Omega^{\varepsilon}}u_{i}v_{i}dx,\quad u=\left(u_{1},...,u_{N}\right),v=\left(v_{1},...,v_{N}\right)\in\mathcal{H}\left(\Omega^{\varepsilon}\right),
\]
\[
\left\langle u,v\right\rangle _{\mathcal{V}^{\varepsilon}}:=\sum_{i=1}^{N}\sum_{j=1}^{n}\int_{\Omega^{\varepsilon}}\frac{\partial u_{i}}{\partial x_{j}}\frac{\partial v_{i}}{\partial x_{j}}dx,\quad u=\left(u_{1},...,u_{N}\right),v=\left(v_{1},...,v_{N}\right)\in\mathcal{V}^{\varepsilon}.
\]

Furthermore, the notation $\mathcal{H}\left(\Gamma^{\varepsilon}\right)$
indicates the corresponding  product of $L^{2}\left(\Gamma^{\varepsilon}\right)$ spaces. 
For $q\in\left(2,\infty\right]$,  the following spaces are also
used
\[
\mathcal{W}^{q}\left(\Omega^{\varepsilon}\right)=L^{q}\left(\Omega^{\varepsilon}\right)\times...\times L^{q}\left(\Omega^{\varepsilon}\right),
\]
\[
\mathcal{W}^{q}\left(\Gamma^{\varepsilon}\right)=L^{q}\left(\Gamma^{\varepsilon}\right)\times...\times L^{q}\left(\Gamma^{\varepsilon}\right).
\]

\section{Well-posedness of the microscopic model}\label{WP}

Before studying the asymptotics behaviour as $\varepsilon\to 0$ (the homogenization limit), we must ensure the well-posedness of the microstructure model. In this section we focus only on the weak solvability of the problem, the stability with respect to the initial data and all parameter being straightforward to prove.  
We remark at this stage that
the structure of the model equation has attracted much attention. For
example, Amann used in \cite{Amann72} the method of sub- and super-
solutions to prove the existence of positive solutions when a 
Robin boundary condition is considered. Brezis and Oswald introduced
in \cite{BO86} an energy minimization approach  to guarantee the existence, uniqueness and positivity
results for the semi-linear elliptic problem with zero Dirichlet boundary
conditions. Very recently, Garc\'ia-Meli\'an et al. \cite{GMRL09}
extended the result in \cite{BO86} (and also of other previous works including 
\cite{CCLG02,CP03}) to problems involving nonlinear boundary conditions
of mixed type. For what we are concerned here,  we will  use Moser-like iterations technique (see the original works by Moser \cite{Mos60,Mos64})
to prove $L^{\infty}$-bounds for all concentrations  and then follow the strategy provided
by Brezis and Oswald \cite{BO86} to study the well-posedness of $\left(P^{\varepsilon}\right)$.

\begin{defn}
A function $u^{\varepsilon}\in\mathcal{V}^{\varepsilon}$ is a weak
solution to $\left(P^{\varepsilon}\right)$ provided that
\begin{equation}
\sum_{i=1}^{N}\int_{\Omega^{\varepsilon}}\left(d_{i}^{\varepsilon}\nabla u_{i}^{\varepsilon}\nabla\varphi_{i}-R_{i}\left(u^{\varepsilon}\right)\varphi_{i}\right)dx-\sum_{i=1}^{N}\varepsilon\int_{\Gamma^{\varepsilon}}\left(a_{i}^{\varepsilon}u_{i}^{\varepsilon}-b_{i}^{\varepsilon}F_{i}\left(u_{i}^{\varepsilon}\right)\right)\varphi_{i}dS_{\varepsilon}=0\quad\mbox{for all}\;\varphi\in\mathcal{V}^{\varepsilon}.\label{eq:weaksol}
\end{equation}

\end{defn}

\begin{defn}
By means of the usual variational characterization, the principal
eigenvalue of $\left(P^{\varepsilon}\right)$ is defined by
\begin{equation}
\lambda_{1}\left(p^{\varepsilon},q^{\varepsilon}\right):=\inf_{u^{\varepsilon}\in\mathcal{V}^{\varepsilon},{\displaystyle \sum_{i=1}^{N}}\left|u_{i}^{\varepsilon}\right|^{2}\neq0}\frac{{\displaystyle \sum_{i=1}^{N}\left(\alpha\int_{\Omega^{\varepsilon}}\left|\nabla_{x}u_{i}^{\varepsilon}\right|^{2}dx-N\int_{\Omega^{\varepsilon}}p_{i}^{\varepsilon}\left|u_{i}^{\varepsilon}\right|^{2}dx-N\int_{\Gamma^{\varepsilon}}q_{i}^{\varepsilon}\left|u_{i}^{\varepsilon}\right|^{2}dS_{\varepsilon}\right)}}{{\displaystyle \sum_{i=1}^{N}\int_{\Omega^{\varepsilon}}\left|u_{i}^{\varepsilon}\right|^{2}dx}},\label{eq:princi-eigen}
\end{equation}
where $p_{i}^{\varepsilon}$ and $q_{i}^{\varepsilon}$ are measurable
such that either they are simultaneously bounded from above or from
below (this leads to $\lambda_{1}\in\left(-\infty,\infty\right]$
or $\lambda_{1}\in\left[-\infty,\infty\right)$, correspondingly).
Here, we denote $\alpha:=\min\left\{ \alpha_{1},...,\alpha_{N}\right\} $.\end{defn}
\begin{lem}
\label{lem:3}Assume $\left(\mbox{A}_{1}\right)$-$\left(\mbox{A}_{5}\right)$
and replace $\left(\mbox{A}_{4}\right)$ by $\left(\mbox{A}_{6}\right)$.
Let $u^{\varepsilon}\in\mathcal{V}^{\varepsilon}\cap\mathcal{H}\left(\Gamma^{\varepsilon}\right)$
be a weak solution to $\left(P^{\varepsilon}\right)$, then $u^{\varepsilon}\in\mathcal{W}^{\infty}\left(\Omega^{\varepsilon}\right)$
and it exists an $\varepsilon$-independent constant $C>0$ such
that
\[
\left\Vert u^{\varepsilon}\right\Vert _{\mathcal{W}^{\infty}\left(\Omega^{\varepsilon}\right)}\le C\left(1+\left\Vert u^{\varepsilon}\right\Vert _{\mathcal{H}\left(\Omega^{\varepsilon}\right)}+\left\Vert u^{\varepsilon}\right\Vert _{\mathcal{H}\left(\Gamma^{\varepsilon}\right)}\right).
\]
\end{lem}
\begin{proof}
Let $\beta\ge1$ and $k_{i}>1$ for all $i=\overline{1,N}$. We begin
by introducing a vector $\varphi^{\varepsilon}$ of test functions
$\varphi_{i}^{\varepsilon}=\min\left\{ v_{i}^{\beta+\frac{1}{2}},k_{i}^{\beta+\frac{1}{2}}\right\} -1$
where $v_{i}=u_{i}^{\varepsilon}+1$ with $u_{i}^{\varepsilon}$ as
in \eqref{eq:weaksol}. Thus, it is straightforward to show that $\varphi^{\varepsilon}\in\mathcal{V}^{\varepsilon}\cap\mathcal{H}\left(\Gamma^{\varepsilon}\right)$.
We have
\begin{eqnarray}
\alpha\left(\beta+\frac{1}{2}\right)\sum_{i=1}^{N}\int_{\left\{ v_{i}<k_{i}\right\} }v_{i}^{\beta-\frac{1}{2}}\left|\nabla v_{i}\right|^{2} & \le & \sum_{i=1}^{N}\int_{\Omega^{\varepsilon}}R_{i}\left(x,u^{\varepsilon}\right)\varphi_{i}^{\varepsilon}dx+\sum_{i=1}^{N}\int_{\Gamma^{\varepsilon}}F_{i}\left(x,u_{i}^{\varepsilon}\right)\varphi_{i}^{\varepsilon}dS_{\varepsilon}\nonumber \\
 & \le & C\sum_{i=1}^{N}\int_{\Omega^{\varepsilon}}\left|1+u_{i}^{\varepsilon}\right|^{2}v_{i}^{\beta+\frac{1}{2}}dx\nonumber \\
 &  & +C\sum_{i=1}^{N}\int_{\Gamma^{\varepsilon}}\left|1+u_{i}^{\varepsilon}\right|^{2}v_{i}^{\beta+\frac{1}{2}}dS_{\varepsilon}\nonumber \\
 & \le & C\sum_{i=1}^{N}\left(\int_{\Omega^{\varepsilon}}v_{i}^{\beta+\frac{3}{2}}dx+\int_{\Gamma^{\varepsilon}}v_{i}^{\beta+\frac{3}{2}}dS_{\varepsilon}\right),\label{eq:3.5}
\end{eqnarray}
where we have used \eqref{eq:c1} and \eqref{eq:c2}.

Now, for every $i\in\left\{ 1,...,N\right\} $, if we assign $\psi_{i}=\min\left\{ v_{i}^{\frac{\beta+\frac{3}{2}}{2}},k_{i}^{\frac{\beta+\frac{3}{2}}{2}}\right\} $,
then one has 
\begin{equation}
\left(\beta+\frac{1}{2}\right)v_{i}^{\beta-\frac{1}{2}}\left|\nabla v_{i}\right|^{2}\chi_{\left\{ v_{i}<k_{i}\right\} }=\frac{4\left(\beta+\frac{1}{2}\right)}{\left(\beta+\frac{3}{2}\right)^{2}}\left|\nabla\psi_{i}\right|^{2}.\label{eq:3.6-1}
\end{equation}

Since $\Omega^{\varepsilon}$ is a Lipschitz domain, then the trace
embedding $H^{1}\left(\Omega^{\varepsilon}\right)\subset L^{q}\left(\partial\Omega^{\varepsilon}\right)$
holds for $1\le q\le2_{\partial\Omega^{\varepsilon}}^{*}$, where
$2_{\partial\Omega^{\varepsilon}}^{*}=2\left(d-1\right)/\left(d-2\right)$
if $d\ge3$, and $2_{\partial\Omega^{\varepsilon}}^{*}=\infty$ if
$d=2$ (cf. \cite{Fan08}). Therefore, given $q\in\left(2,2^{*}\right]$
we apply this embedding to \eqref{eq:3.5} with the aid of \eqref{eq:3.6-1}
and then obtain
\begin{equation}
\frac{4\alpha\left(\beta+\frac{1}{2}\right)}{\left(\beta+\frac{3}{2}\right)^{2}}\sum_{i=1}^{N}\left[\left(\int_{\Gamma^{\varepsilon}}\left|\psi_{i}\right|^{q}dS_{\varepsilon}\right)^{\frac{2}{q}}-\int_{\Omega^{\varepsilon}}\left|\psi_{i}\right|^{2}dx\right]\le C\sum_{i=1}^{N}\left(\int_{\Omega^{\varepsilon}}v_{i}^{\beta+\frac{3}{2}}dx+\int_{\Gamma^{\varepsilon}}v_{i}^{\beta+\frac{3}{2}}dS_{\varepsilon}\right).\label{eq:3.7-1}
\end{equation}

We see that $\psi_{i}^{2}\le v^{\beta+\frac{3}{2}}$ and also
\[
\frac{1}{\beta+\frac{3}{2}}\le\frac{4\left(\beta+\frac{1}{2}\right)}{\left(\beta+\frac{3}{2}\right)^{2}}\le4
\]
holds for all $\beta\ge1$. As a result, \eqref{eq:3.7-1} yields
\begin{equation}
\sum_{i=1}^{N}\left(\int_{\Gamma^{\varepsilon}}\left|\psi_{i}\right|^{q}dS_{\varepsilon}\right)^{\frac{2}{q}}\le C\alpha^{-1}\left(\beta+\frac{3}{2}\right)\sum_{i=1}^{N}\left(\int_{\Omega^{\varepsilon}}v_{i}^{\beta+\frac{3}{2}}dx+\int_{\Gamma^{\varepsilon}}v_{i}^{\beta+\frac{3}{2}}dS_{\varepsilon}\right).\label{eq:3.8}
\end{equation}

Our next aim is to show that if for some $s\ge2$ we have $u^{\varepsilon}\in\mathcal{W}^{s}\left(\Omega^{\varepsilon}\right)\cap\mathcal{W}^{s}\left(\Gamma^{\varepsilon}\right)$,
then $u^{\varepsilon}\in\mathcal{W}^{ks}\left(\Omega^{\varepsilon}\right)\cap\mathcal{W}^{ks}\left(\Gamma^{\varepsilon}\right)$
for $k>1$ arbitrary at each $\varepsilon$-level. In fact, assume
that $u^{\varepsilon}\in\mathcal{W}^{\beta+\frac{3}{2}}\left(\Omega^{\varepsilon}\right)\cap\mathcal{W}^{\beta+\frac{3}{2}}\left(\Gamma^{\varepsilon}\right)$
then letting $k\to\infty$ in \eqref{eq:3.8} gives
\begin{equation}
\sum_{i=1}^{N}\left(\int_{\Gamma^{\varepsilon}}\left|v_{i}\right|^{\frac{q}{2}\left(\beta+\frac{3}{2}\right)}dS_{\varepsilon}\right)^{\frac{2}{q}}\le C\left(\beta+\frac{3}{2}\right)\sum_{i=1}^{N}\left(\int_{\Omega^{\varepsilon}}v_{i}^{\beta+\frac{3}{2}}dx+\int_{\Gamma^{\varepsilon}}v_{i}^{\beta+\frac{3}{2}}dS_{\varepsilon}\right).\label{eq:3.9-2}
\end{equation}

One obtains in the same manner that by the embedding $H^{1}\left(\Omega^{\varepsilon}\right)\subset L^{q}\left(\Omega^{\varepsilon}\right)$
(this is valid for $1\le q\le2_{\Omega^{\varepsilon}}^{*}$ where
$2_{\Omega^{\varepsilon}}^{*}=2d/\left(d-2\right)$ if $d\ge3$, and
$2_{\Omega^{\varepsilon}}^{*}=\infty$ if $d=2$; thus $q$ given
before is definitely valid), we are led to the following estimate
\begin{equation}
\sum_{i=1}^{N}\left(\int_{\Omega^{\varepsilon}}\left|v_{i}\right|^{\frac{q}{2}\left(\beta+\frac{3}{2}\right)}dx\right)^{\frac{2}{q}}\le C\left(\beta+\frac{3}{2}\right)\sum_{i=1}^{N}\left(\int_{\Omega^{\varepsilon}}v_{i}^{\beta+\frac{3}{2}}dx+\int_{\Gamma^{\varepsilon}}v_{i}^{\beta+\frac{3}{2}}dS_{\varepsilon}\right).\label{eq:3.10}
\end{equation}

Combining \eqref{eq:3.9-2}, \eqref{eq:3.10} and the Minkowski inequality
enables us to get
\[
\left(\int_{\Omega^{\varepsilon}}\left|v_{i}\right|^{\frac{q}{2}\left(\beta+\frac{3}{2}\right)}dx+\int_{\Gamma^{\varepsilon}}\left|v_{i}\right|^{\frac{q}{2}\left(\beta+\frac{3}{2}\right)}dS_{\varepsilon}\right)^{\frac{2}{q}}\le C\left(\beta+\frac{3}{2}\right)\sum_{i=1}^{N}\left(\int_{\Omega^{\varepsilon}}v_{i}^{\beta+\frac{3}{2}}dx+\int_{\Gamma^{\varepsilon}}v_{i}^{\beta+\frac{3}{2}}dS_{\varepsilon}\right),
\]
for all $i\in\left\{ 1,...,N\right\} $, which easily leads to, by
raising to the power $1/\left(\beta+\frac{3}{2}\right)$, the fact
that $u_{i}^{\varepsilon}\in L^{\frac{q}{2}\left(\beta+\frac{3}{2}\right)}\left(\Omega^{\varepsilon}\right)\cap L^{\frac{q}{2}\left(\beta+\frac{3}{2}\right)}\left(\Gamma^{\varepsilon}\right)$
for all $i\in\left\{ 1,...,N\right\} $; and hence $u^{\varepsilon}\in\mathcal{W}^{\frac{q}{2}\left(\beta+\frac{3}{2}\right)}\left(\Omega^{\varepsilon}\right)\cap\mathcal{W}^{\frac{q}{2}\left(\beta+\frac{3}{2}\right)}\left(\Gamma^{\varepsilon}\right)$.

The constant $k$  is indicated by $q/2>1$. Thus, if we
choose $q$ and $\beta$ such that
\[
\beta+\frac{3}{2}=2\left(\frac{q}{2}\right)^{n}\quad\mbox{for}\;n=0,1,2,...,
\]
and iterating the above estimate, we obtain, by induction, that
\begin{equation}
\left\Vert v\right\Vert _{2\left(\frac{q}{2}\right)^{n}}\le\prod_{j=0}^{n}\left(2\left(\frac{q}{2}\right)^{j}C\right)^{\frac{1}{2}\left(\frac{2}{q}\right)^{j}}\left\Vert v\right\Vert _{2},\label{eq:3.11}
\end{equation}
where we have denoted by
\[
\left\Vert v\right\Vert _{r}:=\sum_{i=1}^{N}\left(\int_{\Omega^{\varepsilon}}\left|v_{i}\right|^{r}dx+\int_{\Gamma^{\varepsilon}}\left|v_{i}\right|^{r}dS_{\varepsilon}\right)^{\frac{1}{r}}.
\]

It is interesting to point out that since the series $\sum_{n=0}^{\infty}\left(\frac{2}{q}\right)^{n}$
and $\sum_{n=0}^{\infty}n\left(\frac{2}{q}\right)^{n}$ are convergent
for $q>2$, we have
\[
\prod_{j=0}^{n}\left(2\left(\frac{q}{2}\right)^{j}C\right)^{\frac{1}{2}\left(\frac{2}{q}\right)^{j}}<\sqrt{\left(2C\right)^{\sum_{n=0}^{\infty}\left(\frac{2}{q}\right)^{n}}q^{\sum_{n=0}^{\infty}n\left(\frac{2}{q}\right)^{n}}}=C.
\]

Therefore, the constant in the right-hand side of \eqref{eq:3.11}
is indeed independent of $n$, and by passing $n\to\infty$ in \eqref{eq:3.11},
i.e. in the inequality,
\[
\left\Vert v\right\Vert _{\mathcal{W}^{2\left(\frac{q}{2}\right)^{n}}\left(\Omega^{\varepsilon}\right)}\le C\left(\left\Vert v\right\Vert _{\mathcal{H}\left(\Omega^{\varepsilon}\right)}+\left\Vert v\right\Vert _{\mathcal{H}\left(\Gamma^{\varepsilon}\right)}\right),
\]
we finally obtain
\[
\left\Vert v\right\Vert _{\mathcal{W}^{\infty}\left(\Omega^{\varepsilon}\right)}\le C\left(\left\Vert v\right\Vert _{\mathcal{H}\left(\Omega^{\varepsilon}\right)}+\left\Vert v\right\Vert _{\mathcal{H}\left(\Gamma^{\varepsilon}\right)}\right).
\]

Consequently, recalling $v_{i}=u_{i}^{\varepsilon}+1$, we have:
\[
\left\Vert u^{\varepsilon}\right\Vert _{\mathcal{W}^{\infty}\left(\Omega^{\varepsilon}\right)}\le C\left(1+\left\Vert u^{\varepsilon}\right\Vert _{\mathcal{H}\left(\Omega^{\varepsilon}\right)}+\left\Vert u^{\varepsilon}\right\Vert _{\mathcal{H}\left(\Gamma^{\varepsilon}\right)}\right).
\]
This step  completes the proof of the lemma.
\end{proof}

\begin{rem}
Using the trace inequality (cf. \cite[Lemma 2.31]{CP99}) and the norm 
equivalence between $V^{\varepsilon}$ and $H^{1}\left(\Omega^{\varepsilon}\right)$,
if $u^{\varepsilon}\in\mathcal{V}^{\varepsilon}$ then the result
in Lemma \ref{lem:3} reads
\begin{eqnarray*}
\left\Vert u^{\varepsilon}\right\Vert _{\mathcal{W}^{\infty}\left(\Omega^{\varepsilon}\right)} & \le & C\left(1+\varepsilon^{-1/2}\left\Vert u^{\varepsilon}\right\Vert _{\mathcal{H}\left(\Omega^{\varepsilon}\right)}+\left\Vert u^{\varepsilon}\right\Vert _{\mathcal{V}^{\varepsilon}}\right)\\
 & \le & C\left(1+\varepsilon^{-1/2}\left\Vert u^{\varepsilon}\right\Vert _{\mathcal{V}^{\varepsilon}}\right).
\end{eqnarray*}

\end{rem}

\begin{lem}
\label{lem:5}Assume $\left(\mbox{A}_{1}\right)$-$\left(\mbox{A}_{5}\right)$
and that $\lambda_{1}\left(r_{\infty}^{\varepsilon},f_{\infty}^{\varepsilon}\right)>0$
and $\lambda_{1}\left(r_{0}^{\varepsilon},f_{0}^{\varepsilon}\right)<0$
hold. We define the following functional
\[
J\left[u^{\varepsilon}\right]:=\frac{1}{2}\sum_{i=1}^{N}\int_{\Omega^{\varepsilon}}d_{i}^{\varepsilon}\left|\nabla u_{i}^{\varepsilon}\right|^{2}dx-\sum_{i=1}^{N}\int_{\Omega^{\varepsilon}}\mathcal{R}_{i}\left(x,u^{\varepsilon}\right)dx-\sum_{i=1}^{N}\int_{\Gamma^{\varepsilon}}\mathcal{F}_{i}\left(x,u_{i}^{\varepsilon}\right)dS_{\varepsilon},
\]
where
\[
\mathcal{R}_{i}\left(x,u^{\varepsilon}\right):=\int_{0}^{u_{i}^{\varepsilon}}R_{i}\left(x,u_{1}^{\varepsilon},...s_{i},...,u_{N}^{\varepsilon}\right)ds_{i},
\]
\[
\mathcal{F}_{i}\left(x,u_{i}^{\varepsilon}\right):=\int_{0}^{u_{i}^{\varepsilon}}\left(a_{i}^{\varepsilon}s-b_{i}^{\varepsilon}F_{i}\left(s\right)\right)ds,
\]
and the nonlinear terms are extended to be $R_{i}\left(x,0\right)$
and $F_{i}\left(x,0\right)$ for $u_{i}^{\varepsilon}\le0$. Then
$J$ is coercive on $\mathcal{V}^{\varepsilon}$ and lower semi-continuous
for $\mathcal{V}^{\varepsilon}$. Moreover, there exists $\phi\in\mathcal{V}^{\varepsilon}$
such that $J\left[\phi\right]<0$.\end{lem}
\begin{proof}
\textbf{\uline{Step 1:}} (Coerciveness)

Suppose, by contradiction, that it exists a sequence $\left\{ u^{\varepsilon,m}\right\} \subset\mathcal{V}^{\varepsilon}$
such that $\left\Vert u^{\varepsilon,m}\right\Vert _{\mathcal{V}^{\varepsilon}}\to\infty$
while $J\left[u^{\varepsilon,m}\right]\le C$. Setting
\begin{equation}
s_{i,m}=\left(\int_{\Gamma^{\varepsilon}}\left|u_{i}^{\varepsilon,m}\right|^{2}dS_{\varepsilon}\right)^{1/2},\quad t_{i,m}=\left(\int_{\Omega^{\varepsilon}}\left|u_{i}^{\varepsilon,m}\right|^{2}dx\right)^{1/2},\label{eq:3.3}
\end{equation}
we say that $\sum_{i=1}^{N}t_{i,m}^{2}\to\infty$ up to a subsequence
as $m\to\infty$. Indeed, the assumption $J\left[u^{\varepsilon,m}\right]\le C$
yields that 
\begin{equation}
\frac{1}{2}\sum_{i=1}^{N}\int_{\Omega^{\varepsilon}}d_{i}^{\varepsilon}\left|\nabla u_{i}^{\varepsilon,m}\right|^{2}dx\le\sum_{i=1}^{N}\int_{\Omega^{\varepsilon}}\mathcal{R}_{i}\left(x,u^{\varepsilon,m}\right)dx+\sum_{i=1}^{N}\int_{\Gamma^{\varepsilon}}\mathcal{F}_{i}\left(x,u_{i}^{\varepsilon,m}\right)dS_{\varepsilon}+C,\label{eq:3.4-2}
\end{equation}
which, in combination with \eqref{eq:3.3} and $\left(\mbox{A}_{4}\right)$,
leads to
\begin{equation}
\frac{1}{2}\sum_{i=1}^{N}\int_{\Omega^{\varepsilon}}d_{i}^{\varepsilon}\left|\nabla u_{i}^{\varepsilon,m}\right|^{2}dx\le C\left(N\right)\left(1+\sum_{i=1}^{N}t_{i,m}^{2}+\sum_{i=1}^{N}s_{i,m}^{2}\right).\label{eq:3.4-1}
\end{equation}

Here, if $\sum_{i=1}^{N}t_{i,m}^{2}$ is convergent, then $\sum_{i=1}^{N}s_{i,m}^{2}$
cannot be bounded. While putting
\[
v_{i,m}=u_{i}^{\varepsilon,m}/{\displaystyle \sum_{i=1}^{N}s_{i,m}},
\]
it enables us to derive that
\begin{equation}
\sum_{i=1}^{N}\int_{\Omega^{\varepsilon}}\left|\nabla v_{i,m}\right|^{2}dx=\frac{{\displaystyle \sum_{i=1}^{N}\int_{\Omega^{\varepsilon}}\left|\nabla u_{i}^{\varepsilon,m}\right|^{2}}dx}{{\displaystyle \left(\sum_{i=1}^{N}s_{i,m}\right)^{2}}}\le\frac{{\displaystyle \sum_{i=1}^{N}\int_{\Omega^{\varepsilon}}\left|\nabla u_{i}^{\varepsilon,m}\right|^{2}dx}}{{\displaystyle \sum_{i=1}^{N}s_{i,m}^{2}}}.\label{eq:3.4}
\end{equation}

If we assign $\alpha:=\min\left\{ \alpha_{1},...,\alpha_{N}\right\} >0$,
then it follows from \eqref{eq:3.4-1} and \eqref{eq:3.4} that 
\begin{eqnarray*}
\frac{\alpha}{2}\sum_{i=1}^{N}\int_{\Omega^{\varepsilon}}\left|\nabla v_{i,m}\right|^{2}dx & \le & \frac{{\displaystyle \sum_{i=1}^{N}\int_{\Omega^{\varepsilon}}d_{i}^{\varepsilon}\left|\nabla u_{i}^{\varepsilon,m}\right|^{2}dx}}{2{\displaystyle \sum_{i=1}^{N}s_{i,m}^{2}}}\\
 & \le & C\left(N\right)\left(1+\frac{{\displaystyle \sum_{i=1}^{N}t_{i,m}^{2}}}{{\displaystyle \sum_{i=1}^{N}s_{i,m}^{2}}}+\frac{1}{{\displaystyle \sum_{i=1}^{N}s_{i,m}^{2}}}\right)\le C\left(N\right).
\end{eqnarray*}

Now, we claim that there exists $v_{i}\in V^{\varepsilon}$ such that
$v_{i,m}\rightharpoonup v_{i}$ weakly in $V^{\varepsilon}$, and
then strongly in $L^{2}\left(\Omega^{\varepsilon}\right)$ and in
$L^{2}\left(\Gamma^{\varepsilon}\right)$. However, it implies here
a contradiction. It is because  we have $v_{i}\equiv0$ in $\Omega^{\varepsilon}$
for all $i=\overline{1,N}$ while
\[
\sum_{i=1}^{N}\int_{\Gamma^{\varepsilon}}\left|v_{i}\right|^{2}dS_{\varepsilon}=\left(\sum_{i=1}^{N}s_{i}\right)^{-2}\sum_{i=1}^{N}\int_{\Gamma^{\varepsilon}}\left|u_{i}^{\varepsilon}\right|^{2}dS_{\varepsilon}\ge N^{-1}>0.
\]

Let us now assume that $\sum_{i=1}^{N}t_{i,m}^{2}$ is divergent.
By putting
\[
w_{i,m}=u_{i}^{\varepsilon,m}/{\displaystyle \sum_{i=1}^{N}t_{i,m}},
\]
we have, in the same manner, that
\[
\frac{\alpha}{2}\sum_{i=1}^{N}\int_{\Omega^{\varepsilon}}\left|\nabla w_{i,m}\right|^{2}dx\le C\left(N\right)\left(1+\frac{1}{{\displaystyle \sum_{i=1}^{N}t_{i,m}^{2}}}+\frac{{\displaystyle \sum_{i=1}^{N}s_{i,m}^{2}}}{{\displaystyle \sum_{i=1}^{N}t_{i,m}^{2}}}\right).
\]

From \eqref{eq:3.3}, we know that
\begin{equation}
\sum_{i=1}^{N}\int_{\Omega^{\varepsilon}}\left|w_{i,m}\right|^{2}dx=\left(\sum_{i=1}^{N}t_{i,m}\right)^{-2}\sum_{i=1}^{N}\int_{\Omega^{\varepsilon}}\left|u_{i}^{\varepsilon,m}\right|^{2}dx\le1,\label{eq:3.6}
\end{equation}
and
\begin{equation}
\sum_{i=1}^{N}\int_{\Gamma^{\varepsilon}}\left|w_{i,m}\right|^{2}dS_{\varepsilon}\ge N^{-1}\left(\sum_{i=1}^{N}t_{i,m}^{2}\right)^{-1}\sum_{i=1}^{N}\int_{\Gamma^{\varepsilon}}\left|u_{i}^{\varepsilon,m}\right|^{2}dS_{\varepsilon}\ge\frac{{\displaystyle \sum_{i=1}^{N}s_{i,m}^{2}}}{N{\displaystyle \sum_{i=1}^{N}t_{i,m}^{2}}}.\label{eq:3.7}
\end{equation}

Combining the trace inequality (cf. \cite[Lemma 2.31]{CP99}) with
\eqref{eq:3.6} and \eqref{eq:3.7}, we obtain
\begin{eqnarray*}
\frac{{\displaystyle \sum_{i=1}^{N}s_{i,m}^{2}}}{{\displaystyle \sum_{i=1}^{N}t_{i,m}^{2}}} & \le & N\sum_{i=1}^{N}\int_{\Gamma^{\varepsilon}}\left|w_{i,m}\right|^{2}dS_{\varepsilon}\\
 & \le & CN\left(2\sum_{i=1}^{N}\left(\int_{\Omega^{\varepsilon}}\left|w_{i,m}\right|^{2}dx\right)^{1/2}\left(\int_{\Omega^{\varepsilon}}\left|\nabla w_{i,m}\right|^{2}dx\right)^{1/2}+\varepsilon^{-1}\sum_{i=1}^{N}\int_{\Omega^{\varepsilon}}\left|w_{i,m}\right|^{2}dx\right)\\
 & \le & CN\left(2\sum_{i=1}^{N}\left(\int_{\Omega^{\varepsilon}}\left|\nabla w_{i,m}\right|^{2}dx\right)^{1/2}+\varepsilon^{-1}\right).
\end{eqnarray*}

It yields that
\[
\sum_{i=1}^{N}\int_{\Omega^{\varepsilon}}\left|\nabla w_{i,m}\right|^{2}dx\le\frac{2C\left(N\right)}{\alpha}\left[\sum_{i=1}^{N}\left(\int_{\Omega^{\varepsilon}}\left|\nabla w_{i,m}\right|^{2}dx\right)^{1/2}+C\left(\varepsilon\right)+\left(\sum_{i=1}^{N}t_{i,m}^{2}\right)^{-1}\right],
\]
which finally leads to
\begin{equation}
\left|\left(\int_{\Omega^{\varepsilon}}\left|\nabla w_{i,m}\right|^{2}dx\right)^{1/2}-\frac{C\left(N\right)}{\alpha}\right|\le C\left(N,\varepsilon\right)\left(1+\left(\sum_{i=1}^{N}t_{i,m}^{2}\right)^{-1}\right)^{1/2}\;\mbox{for all}\;i=\overline{1,N}.\label{eq:3.9-1}
\end{equation}

Therefore, $\int_{\Omega^{\varepsilon}}\left|\nabla w_{i,m}\right|^{2}dx$
is bounded by the inequality \eqref{eq:3.9-1}. So, up to a subsequence,
$w_{i,m}\rightharpoonup w_{i}$ weakly in $V^{\varepsilon}$, and
then strongly in $L^{2}\left(\Omega^{\varepsilon}\right)$ and $L^{2}\left(\Gamma^{\varepsilon}\right)$.
In addition, it can be proved that $\sum_{i=1}^{N}\int_{\Omega^{\varepsilon}}\left|w_{i}\right|^{2}dx\ge N^{-1}>0$,
and from \eqref{eq:3.4-2}, it gives us that
\begin{equation}
\frac{\alpha}{2}\sum_{i=1}^{N}\int_{\Omega^{\varepsilon}}\left|\nabla w_{i,m}\right|^{2}dx\le\frac{C}{{\displaystyle \sum_{i=1}^{N}t_{i,m}^{2}}}+\sum_{i=1}^{N}\int_{\Omega^{\varepsilon}}\frac{\mathcal{R}_{i}\left(x,u^{\varepsilon,m}\right)}{{\displaystyle \sum_{i=1}^{N}t_{i,m}^{2}}}dx+\sum_{i=1}^{N}\int_{\Gamma^{\varepsilon}}\frac{\mathcal{F}_{i}\left(x,u_{i}^{\varepsilon,m}\right)}{{\displaystyle \sum_{i=1}^{N}t_{i,m}^{2}}}dS_{\varepsilon}.\label{eq:3.9}
\end{equation}

We now consider the second integral on the right-hand side of
the above inequality, then the third one is totally similar. Using
the fact that $w_{i,m}\to w_{i}$ strongly in $L^{2}\left(\Omega^{\varepsilon}\right)$
and the assumptions $\left(\mbox{A}_{4}\right)$-$\left(\mbox{A}_{5}\right)$
in combination with the Fatou lemma, we get
\[
\limsup_{m\to\infty}\sum_{i=1}^{N}\int_{\Omega^{\varepsilon}}\frac{\mathcal{R}_{i}\left(x,u^{\varepsilon,m}\right)}{{\displaystyle \sum_{i=1}^{N}t_{i,m}^{2}}}dx\le\frac{N}{2}\sum_{i=1}^{N}\int_{\Omega^{\varepsilon}\cap\left\{ w>0\right\} }r_{\infty,i}^{\varepsilon}\left|w_{i,m}\right|^{2}dx,
\]
where we have also applied the following inequalities
\[
N^{-1}\left(\sum_{i=1}^{N}t_{i,m}^{2}\right)^{-1}\le\left(\sum_{i=1}^{N}t_{i,m}\right)^{-2}\le\left|w_{i,m}\right|^{2}\left|u_{i}^{\varepsilon,m}\right|^{-2},
\]
\[
\limsup_{u_{i}^{\varepsilon}\to\infty}\frac{\mathcal{R}_{i}\left(x,u^{\varepsilon}\right)}{\left|u_{i}^{\varepsilon}\right|^{2}}\le\frac{1}{2}r_{\infty,i}^{\varepsilon}\left(x\right)\quad\mbox{for}\;\mbox{a.e.}\;x\in\Omega^{\varepsilon}.
\]

Thus, passing to the limit in \eqref{eq:3.9} we are led to
\[
\frac{\alpha}{2}\sum_{i=1}^{N}\int_{\Omega^{\varepsilon}}\left|\nabla w_{i}\right|^{2}dx\le\frac{N}{2}\left(\sum_{i=1}^{N}\int_{\Omega^{\varepsilon}\cap\left\{ w>0\right\} }r_{\infty,i}^{\varepsilon}\left|w_{i}\right|^{2}dx+\sum_{i=1}^{N}\int_{\Gamma^{\varepsilon}\cap\left\{ w>0\right\} }f_{\infty,i}^{\varepsilon}\left|w_{i}\right|^{2}dS_{\varepsilon}\right).
\]

Recall that $\lambda_{1}\left(r_{\infty}^{\varepsilon},f_{\infty}^{\varepsilon}\right)>0$,
it then gives us that $w_{i}^{+}\equiv0$ for all $i=\overline{1,N}$.
As a consequence, $w_{i}\equiv0$ while it contradicts the above result
$\sum_{i=1}^{N}\int_{\Omega^{\varepsilon}}\left|w_{i}\right|^{2}dx\ge N^{-1}$.

Hence, $J$ is coercive.

\textbf{\uline{Step 2:}} (Lower semi-continuity)

It can be proved as in \cite{BO86,GMRL09} that: if  $u^{\varepsilon,m}\rightharpoonup u^{\varepsilon}$
in $\mathcal{V}^{\varepsilon}$, then we obtain
\[
\limsup_{m\to\infty}\int_{\Omega^{\varepsilon}}\mathcal{R}_{i}\left(x,u^{\varepsilon,m}\right)dx\le\int_{\Omega^{\varepsilon}}\mathcal{R}_{i}\left(x,u^{\varepsilon}\right)dx,
\]
\[
\limsup_{m\to\infty}\int_{\Gamma^{\varepsilon}}\mathcal{F}_{i}\left(x,u_{i}^{\varepsilon,m}\right)dS_{\varepsilon}\le\int_{\Gamma^{\varepsilon}}\mathcal{F}_{i}\left(x,u_{i}^{\varepsilon}\right)dS_{\varepsilon},
\]
by using the growth assumptions $\left(\mbox{A}_{4}\right)$ in combination
with the Fatou lemma. Thus, $J$ is lower semi-continuous.

This result tells us that $J$ achieves the global minimum at a function
$u^{\varepsilon}\in\mathcal{V}^{\varepsilon}$. If we replace $u^{\varepsilon}$
by $\left(u^{\varepsilon}\right)^{+}$, $u^{\varepsilon}$ can be
supposed to be non-negative. Moreover, the last step shows that $u^{\varepsilon}$
is non-trivial.

\textbf{\uline{Step 3:}} (Non-triviality of the minimisers)

What we need to prove now is that there exists $\phi\in\mathcal{V}^{\varepsilon}$
such that $J\left[\phi\right]<0$. In fact, given $\psi\in\mathcal{V}^{\varepsilon}\cap\mathcal{W}^{\varepsilon}$
satisfying $\left\Vert \psi\right\Vert _{\mathcal{W}^{\varepsilon}}=1$
and
\[
\alpha\sum_{i=1}^{N}\int_{\Omega^{\varepsilon}}\left|\nabla\psi_{i}\right|^{2}dx<N\sum_{i=1}^{N}\left(\int_{\Omega^{\varepsilon}}r_{0,i}^{\varepsilon}\left|\psi_{i}\right|^{2}dx+\int_{\Gamma^{\varepsilon}}f_{0,i}^{\varepsilon}\left|\psi_{i}\right|^{2}dS_{\varepsilon}\right).
\]

In fact, here we assume that $\psi$ is non-negative. By the assumptions
$\left(\mbox{A}_{4}\right)$-$\left(\mbox{A}_{5}\right)$, we have
\[
\liminf_{\delta\to0^{+}}\frac{\mathcal{R}_{i}\left(x,\delta\psi\right)}{\delta^{2}}\ge\frac{1}{2}r_{0,i}^{\varepsilon}\left(x\right)\left|\psi\right|^{2}\ge\frac{1}{2}r_{0,i}^{\varepsilon}\left(x\right)\left|\psi_{i}\right|^{2}\quad\mbox{for}\;\mbox{a.e.}\;x\in\Omega^{\varepsilon},
\]
and
\[
\liminf_{\delta\to0^{+}}\frac{\mathcal{F}_{i}\left(x,\delta\psi_{i}\right)}{\delta^{2}}\ge\frac{1}{2}f_{0,i}^{\varepsilon}\left(x\right)\left|\psi_{i}\right|^{2}\quad\mbox{for}\;\mbox{a.e.}\;x\in\Gamma^{\varepsilon}.
\]

This coupling with the Fatou lemma enable us to obtain the following
\[
\sum_{i=1}^{N}\left(\liminf_{\delta\to0^{+}}\int_{\Omega^{\varepsilon}}\frac{\mathcal{R}_{i}\left(x,\delta\psi\right)}{\delta^{2}}dx+\liminf_{\delta\to0^{+}}\frac{\mathcal{F}_{i}\left(x,\delta\psi_{i}\right)}{\delta^{2}}\right)
\]
\[
\ge\frac{1}{2}\sum_{i=1}^{N}\left(\int_{\Omega^{\varepsilon}}r_{0,i}^{\varepsilon}\left|\psi_{i}\right|^{2}dx+\int_{\Gamma^{\varepsilon}}f_{0,i}^{\varepsilon}\left|\psi_{i}\right|^{2}dS_{\varepsilon}\right),
\]
which leads to
\[
\limsup_{\delta\to0^{+}}\frac{J\left[\delta\psi\right]}{\delta^{2}}<0.
\]

Hence, to complete the proof, we need to choose $\phi=\delta\psi$.\end{proof}
\begin{thm}
\label{thm:6}Assume $\left(\mbox{A}_{1}\right)$-$\left(\mbox{A}_{5}\right)$
and $\lambda_{1}\left(r_{\infty}^{\varepsilon},f_{\infty}^{\varepsilon}\right)>0$,
$\lambda_{1}\left(r_{0}^{\varepsilon},f_{0}^{\varepsilon}\right)<0$
hold. Then $\left(P^{\varepsilon}\right)$ admits at least a non-negative
weak solution $u^{\varepsilon}\in\mathcal{V}^{\varepsilon}\cap\mathcal{W}^{\infty}\left(\Omega^{\varepsilon}\right)$.\end{thm}
\begin{proof}
We begin the proof by introducing the approximate system
\[
\left(P^{k,\varepsilon}\right):\;\begin{cases}
\nabla\cdot\left(-d_{i}^{\varepsilon}\nabla u_{i}^{\varepsilon}\right)=R_{i}^{k}\left(u^{\varepsilon}\right), & \mbox{in}\;\Omega^{\varepsilon}\subset\mathbb{R}^{d},\\
d_{i}^{\varepsilon}\nabla u_{i}^{\varepsilon}\cdot\mbox{n}=G_{i}^{k}\left(u_{i}^{\varepsilon}\right), & \mbox{on}\;\Gamma^{\varepsilon},\\
u_{i}^{\varepsilon}=0, & \mbox{on}\;\Gamma^{ext},
\end{cases}
\]
in which we have defined that for each integer $k>0$ the truncated reaction rates
\[
R_{i}^{k}\left(u^{\varepsilon}\right):=\begin{cases}
\max\left\{ -ku_{i}^{\varepsilon},R_{i}\left(u^{\varepsilon}\right)\right\} , & \mbox{if}\;u_{i}^{\varepsilon}\ge0,\\
R_{i}\left(0\right), & \mbox{if}\;u_{i}^{\varepsilon}<0,
\end{cases}
\]
and
\[
G_{i}^{k}\left(u_{i}^{\varepsilon}\right):=\begin{cases}
\varepsilon\max\left\{ -ku_{i}^{\varepsilon},a_{i}^{\varepsilon}u_{i}^{\varepsilon}-b_{i}^{\varepsilon}F_{i}\left(u_{i}^{\varepsilon}\right)\right\} , & \mbox{if}\;u_{i}^{\varepsilon}\ge0,\\
-\varepsilon b_{i}^{\varepsilon}F_{i}\left(0\right), & \mbox{if}\;u_{i}^{\varepsilon}<0.
\end{cases}
\]

It is easy to check that our truncated functions $R_{i}^{k}$ and
$G_{i}^{k}$ fulfill both $\left(\mbox{A}_{4}\right)$ and $\left(\mbox{A}_{6}\right)$.
In addition, if we set elements $R_{0,i}^{k},R_{\infty,i}^{k},G_{0,i}^{k},G_{\infty,i}^{k}$
as functions in $\left(\mbox{A}_{5}\right)$ by $R_{i}^{k}$ and $G_{i}^{k}$,
one may prove that
\[
r_{0,i}^{\varepsilon}\le R_{0,i}^{k},\;r_{\infty,i}^{\varepsilon}\le R_{\infty,i}^{k},\;f_{0,i}^{\varepsilon}\le G_{0,i}^{k},\;f_{\infty,i}^{\varepsilon}\le G_{\infty,i}^{k}\quad\mbox{for all}\;i\in\left\{ 1,...,N\right\} ,
\]
and $\lambda_{1}\left(R_{0}^{k},G_{0}^{k}\right)<0$ and $\lambda_{1}\left(R_{\infty}^{k},G_{\infty}^{k}\right)>0$
for $k$ large (see, e.g. \cite{GMRL09}).

Thanks to Lemma \ref{lem:5}, the problem $\left(P^{k,\varepsilon}\right)$
admits a global non-trivial and non-negative minimizer, denoted by
$u^{k,\varepsilon}$, which belongs to $\mathcal{V}^{\varepsilon}$
and it is associated with the following functional
\[
J^{k}\left[u^{\varepsilon}\right]:=\frac{1}{2}\sum_{i=1}^{N}\int_{\Omega^{\varepsilon}}d_{i}^{\varepsilon}\left|\nabla u_{i}^{\varepsilon}\right|dx-\sum_{i=1}^{N}\int_{\Omega^{\varepsilon}}\mathcal{R}_{i}^{k}\left(x,u^{\varepsilon}\right)dx-\sum_{i=1}^{N}\int_{\Gamma^{\varepsilon}}\mathcal{F}_{i}^{k}\left(x,u_{i}^{\varepsilon}\right)dS_{\varepsilon}.
\]

Furthermore, $u^{k,\varepsilon}$ defines a weak solution to the problem
$\left(P^{k,\varepsilon}\right)$ for every $k$ and thus, $u^{k,\varepsilon}\in\mathcal{W}^{\infty}\left(\Omega^{\varepsilon}\right)$
by Lemma \ref{lem:3}.

Now, we assign a vector $v^{\varepsilon}$ whose elements are defined
by $v_{i}^{\varepsilon}:=\min\left\{ u_{i}^{\varepsilon},u_{i}^{k,\varepsilon}\right\} $
where $u\in\mathcal{V}^{\varepsilon}$ is the global minimizer constructed
from the functional $J$. We shall prove that $J\left[v^{\varepsilon}\right]\le J\left[u^{\varepsilon}\right]$.
Note that when doing so, $v^{\varepsilon}\in\mathcal{W}^{\infty}\left(\Omega^{\varepsilon}\right)$
and then define a weak solution $u\in\mathcal{V}^{\varepsilon}\cap\mathcal{W}^{\infty}\left(\Omega^{\varepsilon}\right)$
to $\left(P^{\varepsilon}\right)$.

In fact, one has
\[
J^{k}\left[u^{k,\varepsilon}\right]\le J\left[\phi\right]\quad\mbox{for all}\;\phi\in\mathcal{V}^{\varepsilon}.
\]

Then by choosing $\phi$ such that $\phi_{i}:=\max\left\{ u_{i}^{\varepsilon},u_{i}^{k,\varepsilon}\right\} $
we have
\[
\sum_{i=1}^{N}\int_{\left\{ u_{i}^{k,\varepsilon}<u_{i}^{\varepsilon}\right\} \cap\Omega^{\varepsilon}}\left(\frac{1}{2}d_{i}^{\varepsilon}\left|\nabla u_{i}^{k,\varepsilon}\right|^{2}-\mathcal{R}_{i}^{k}\left(x,u^{k,\varepsilon}\right)\right)dx-\sum_{i=1}^{N}\int_{\left\{ u_{i}^{k,\varepsilon}<u_{i}^{\varepsilon}\right\} \cap\Gamma^{\varepsilon}}\mathcal{F}_{i}^{k}\left(x,u_{i}^{k,\varepsilon}\right)dS_{\varepsilon}
\]
\begin{equation}
\le\sum_{i=1}^{N}\int_{\left\{ u_{i}^{k,\varepsilon}<u_{i}^{\varepsilon}\right\} \cap\Omega^{\varepsilon}}\left(\frac{1}{2}d_{i}^{\varepsilon}\left|\nabla u_{i}^{\varepsilon}\right|^{2}-\mathcal{R}_{i}^{k}\left(x,u^{\varepsilon}\right)\right)dx-\sum_{i=1}^{N}\int_{\left\{ u_{i}^{k,\varepsilon}<u_{i}^{\varepsilon}\right\} \cap\Gamma^{\varepsilon}}\mathcal{F}_{i}^{k}\left(x,u_{i}^{\varepsilon}\right)dS_{\varepsilon}.\label{eq:3.18}
\end{equation}

In addition, by the choice of $J$ (see in Lemma \ref{lem:5}) we deduce that
\begin{eqnarray}
J\left[v^{\varepsilon}\right]-J\left[u^{\varepsilon}\right] & = & \sum_{i=1}^{N}\int_{\left\{ u_{i}^{k,\varepsilon}<u_{i}^{\varepsilon}\right\} \cap\Omega^{\varepsilon}}\frac{1}{2}d_{i}^{\varepsilon}\left(\left|\nabla u_{i}^{k,\varepsilon}\right|^{2}-\left|\nabla u_{i}^{\varepsilon}\right|\right)dx\nonumber \\
 &  & -\sum_{i=1}^{N}\int_{\left\{ u_{i}^{k,\varepsilon}<u_{i}^{\varepsilon}\right\} \cap\Omega^{\varepsilon}}\left(\mathcal{R}_{i}\left(x,u^{k,\varepsilon}\right)-\mathcal{R}_{i}\left(x,u^{\varepsilon}\right)\right)dx\nonumber \\
 &  & -\sum_{i=1}^{N}\int_{\left\{ u_{i}^{k,\varepsilon}<u_{i}^{\varepsilon}\right\} \cap\Gamma^{\varepsilon}}\left(\mathcal{F}_{i}\left(x,u_{i}^{k,\varepsilon}\right)-\mathcal{F}_{i}\left(x,u_{i}^{\varepsilon}\right)\right)dS_{\varepsilon}.\label{eq:3.19}
\end{eqnarray}

On the other hand, \eqref{eq:3.18} yields
\begin{equation}
\sum_{i=1}^{N}\int_{\left\{ u_{i}^{k,\varepsilon}<u_{i}^{\varepsilon}\right\} \cap\Omega^{\varepsilon}}\left[\mathcal{R}_{i}^{k}\left(x,u^{k,\varepsilon}\right)-\mathcal{R}_{i}^{k}\left(x,u^{\varepsilon}\right)-\left(\mathcal{R}_{i}\left(x,u^{k,\varepsilon}\right)-\mathcal{R}_{i}\left(x,u^{\varepsilon}\right)\right)\right]\le0,\label{eq:3.20}
\end{equation}
and
\begin{equation}
\sum_{i=1}^{N}\int_{\left\{ u_{i}^{k,\varepsilon}<u_{i}^{\varepsilon}\right\} \cap\Gamma^{\varepsilon}}\left[\mathcal{F}_{i}^{k}\left(x,u^{k,\varepsilon}\right)-\mathcal{F}_{i}^{k}\left(x,u^{\varepsilon}\right)-\left(\mathcal{F}_{i}\left(x,u^{k,\varepsilon}\right)-\mathcal{F}_{i}\left(x,u^{\varepsilon}\right)\right)\right]\le0.\label{eq:3.21}
\end{equation}

Hence, combining \eqref{eq:3.18}-\eqref{eq:3.21} we complete the
proof of the lemma. This tells us that under assumptions $\left(\mbox{A}_{1}\right)$-$\left(\mbox{A}_{5}\right)$
the problem $\left(P^{\varepsilon}\right)$ admits a non-negative,
non-trivial and bounded weak vector of solutions $u^{\varepsilon}$
at each $\varepsilon$-level.\end{proof}
\begin{rem}
If $R_{i}\left(u^{\varepsilon}\right)\ge-Mu_{i}^{\varepsilon}$ in
$\Omega^{\varepsilon}$ (or for each subdomain of $\Omega^{\varepsilon}$
if rigorously stated) for some $\varepsilon$-dependent constant $M>0$
and all $i\in\left\{ 1,...,N\right\} $, then $\left(P^{\varepsilon}\right)$
has at least a positive, non-trivial and bounded weak solution $u^{\varepsilon}$
by the Hopf strong maximum principle. Furthermore, one may prove in
the same vein in \cite[Lemma 13]{GMRL09} that the solution is unique
by using vectors of test functions $\varphi_{\delta}^{\varepsilon}$
and $\psi_{\delta}^{\varepsilon}$ whose elements are given by 
\[
\varphi_{\delta,i}^{\varepsilon}=\frac{\left(u_{i}^{\varepsilon}+\delta\right)^{2}-\left(v_{i}^{\varepsilon}+\delta\right)^{2}}{u_{i}^{\varepsilon}+\delta},\;\psi_{\delta,i}^{\varepsilon}=\frac{\left(u_{i}^{\varepsilon}+\delta\right)^{2}-\left(v_{i}^{\varepsilon}+\delta\right)^{2}}{v_{i}^{\varepsilon}+\delta},
\]
where $u_{i}^{\varepsilon}$ and $v_{i}^{\varepsilon}$ are two solutions
of $\left(P^{\varepsilon}\right)$ at each layer $i\in\left\{ 1,...,N\right\} $,
which are expected to equal to each other.
\end{rem}

\begin{rem}
In the case of zero Neumann boundary condition on $\Gamma^{\varepsilon}$,
if the nonlinearity $R_{i}$ is globally Lipschitzi with the Lipschitz
constant, denoted by $L_{i}$, independent of the scale $\varepsilon$
for any $i\in\left\{ 1,...,N\right\} $, then we may use an iterative scheme  to deal with the existence and uniqueness of
solutions to our problem. In fact, for $n\in\mathbb{N}$
such an iterative scheme is given by 
\begin{equation}
\left(P_{n}^{\varepsilon}\right):\;\begin{cases}
\nabla\cdot\left(-d_{i}^{\varepsilon}\nabla u_{i}^{\varepsilon,n+1}\right)=R_{i}\left(u^{\varepsilon,n}\right), & \mbox{in}\;\Omega^{\varepsilon},\\
d_{i}^{\varepsilon}\nabla u_{i}^{\varepsilon,n+1}\cdot\mbox{n}=0, & \mbox{on}\;\Gamma^{\varepsilon},\\
u_{i}^{\varepsilon,n+1}=0, & \mbox{on}\;\Gamma^{ext},
\end{cases}\label{eq:Pn}
\end{equation}
where the starting point is $u^{\varepsilon,0}=0$.

This global Lipschitz assumption is an alternative to $\left(\mbox{A}_{4}\right)$
for $R_{i}$ and it is termed as $\left(\mbox{A}_{4}'\right)$.\end{rem}
\begin{thm}
\label{thm:9}Assume $\left(\mbox{A}_{1}\right)$ and $\left(\mbox{A}_{3}\right)$
hold (without $F_{i}$) and suppose that the nonlinearity $R_{i}$
satisfy $\left(\mbox{A}_{4}'\right)$ replaced by $\left(\mbox{A}_{4}\right)$.
Then, the problem $\left(P^{\varepsilon}\right)$ with zero Neumann
boundary condition on $\Gamma^{\varepsilon}$ has a unique solution
in $\mathcal{V}^{\varepsilon}$ if the constant $\alpha^{-1}\max_{1\le i\le N}\left\{ L_{i}\right\} N$
is small enough.\end{thm}
\begin{proof}
It is worth noting that the problem \eqref{eq:Pn} admits a unique
solution in $\mathcal{V}^{\varepsilon}$ for any $n$. Then, the functional
$w_{i}^{\varepsilon,n}=u_{i}^{\varepsilon,n+1}-u_{i}^{\varepsilon,n}\in V^{\varepsilon}$
satisfies the following problem:
\[
\begin{cases}
\nabla\cdot\left(-d_{i}^{\varepsilon}\nabla w_{i}^{\varepsilon,n}\right)=R_{i}\left(u^{\varepsilon,n}\right)-R_{i}\left(u^{\varepsilon,n-1}\right), & \mbox{in}\;\Omega^{\varepsilon},\\
d_{i}^{\varepsilon}\nabla w_{i}^{\varepsilon,n}\cdot\mbox{n}=0, & \mbox{on}\;\Gamma^{\varepsilon},\\
w_{i}^{\varepsilon,n}=0. & \mbox{on}\;\Gamma^{ext},
\end{cases}
\]

Using the test function $\psi_{i}\in V^{\varepsilon}$ we arrive at
\[
\left\langle d_{i}^{\varepsilon}w_{i}^{\varepsilon,n},\psi_{i}\right\rangle _{V^{\varepsilon}}=\left\langle R_{i}\left(u^{\varepsilon,n}\right)-R_{i}\left(u^{\varepsilon,n-1}\right),\psi_{i}\right\rangle _{L^{2}\left(\Omega^{\varepsilon}\right)}
\]

We may consider an estimate for the above expression:
\begin{equation}
\alpha\sum_{i=1}^{N}\left|\left\langle w_{i}^{\varepsilon,n},\psi_{i}\right\rangle _{V^{\varepsilon}}\right|\le\sum_{i=1}^{N}L_{i}N\left|\left\langle w_{i}^{\varepsilon,n-1},\psi_{i}\right\rangle _{L^{2}\left(\Omega^{\varepsilon}\right)}\right|\label{eq:3.23}
\end{equation}

Thanks to H\"older's and Poincar\'e inequalities, we have
\[
\sum_{i=1}^{N}\left|\left\langle w_{i}^{\varepsilon,n},\psi_{i}\right\rangle _{V^{\varepsilon}}\right|\le C_{p}\alpha^{-1}\max_{1\le i\le N}\left\{ L_{i}\right\} N\left\Vert w^{\varepsilon,n-1}\right\Vert _{\mathcal{V}^{\varepsilon}}\left\Vert \psi\right\Vert _{\mathcal{V}^{\varepsilon}},
\]
where $C_{p}>0$ is the Poincar\'e constant independent of the choice
of $\varepsilon$, but the dimension $d$ of the media (see, e.g.
\cite[Lemma 2.1]{CP99} and \cite{DD02}).

At this point, if the constant $\alpha^{-1}\max_{1\le i\le N}\left\{ L_{i}\right\} N$
is small enough such that $\kappa_{p}:=C_{p}\alpha^{-1}\max_{1\le i\le N}\left\{ L_{i}\right\} N<1$,
then choosing $\psi_{i}=w_{i}^{\varepsilon,n}$ for $i\in\left\{ 1,...,N\right\} $
we obtain that
\[
\left\Vert w^{\varepsilon,n}\right\Vert _{\mathcal{V}^{\varepsilon}}\le\kappa_{p}\left\Vert w^{\varepsilon,n-1}\right\Vert _{\mathcal{V}^{\varepsilon}}.
\]

Consequently, for some $k\in\mathbb{N}$ we get
\begin{eqnarray}
\left\Vert u^{\varepsilon,n+k}-u^{\varepsilon,n}\right\Vert _{\mathcal{V}^{\varepsilon}} & \le & \left\Vert u^{\varepsilon,n+k}-u^{\varepsilon,n+k-1}\right\Vert _{\mathcal{V}^{\varepsilon}}+...+\left\Vert u^{\varepsilon,n+1}-u^{\varepsilon,n}\right\Vert _{\mathcal{V}^{\varepsilon}}\nonumber \\
 & \le & \kappa_{p}^{n+k-1}\left\Vert u^{\varepsilon,1}-u^{\varepsilon,0}\right\Vert _{\mathcal{V}^{\varepsilon}}+...+\kappa_{p}^{n}\left\Vert u^{\varepsilon,1}-u^{\varepsilon,0}\right\Vert _{\mathcal{V}^{\varepsilon}}\nonumber \\
 & \le & \kappa_{p}^{n}\left(\kappa_{p}^{k-1}+\kappa_{p}^{k-2}+...+1\right)\left\Vert u^{\varepsilon,1}\right\Vert _{\mathcal{V}^{\varepsilon}}\nonumber \\
 & \le & \frac{\kappa_{p}^{n}\left(1-\kappa_{p}^{k}\right)}{1-\kappa_{p}}\left\Vert u^{\varepsilon,1}\right\Vert _{\mathcal{V}^{\varepsilon}}.\label{eq:uu}
\end{eqnarray}

Therefore, $\left\{ u^{\varepsilon,n}\right\} $ is a Cauchy sequence
in $\mathcal{V}^{\varepsilon}$, and then there exists uniquely $u^{\varepsilon}\in\mathcal{V}^{\varepsilon}$
such that $u^{\varepsilon,n}\to u^{\varepsilon}$ strongly in $\mathcal{V}^{\varepsilon}$
as $n\to\infty$. Remarkably, this convergence combining with the
Lipschitz property of $R_{i}$ leads to the fact that $R_{i}\left(u^{\varepsilon,n}\right)\to R_{i}\left(u^{\varepsilon}\right)$
strongly in $\mathcal{V}^{\varepsilon}$ as $n\to\infty$. As a result,
the function $u^{\varepsilon}$ is the solution of the problem $\left(P^{\varepsilon}\right)$
when passing to the limit in $n$.

In addition, when $k\to\infty$, it follows from \eqref{eq:uu} that
\[
\left\Vert u^{\varepsilon,n}-u^{\varepsilon}\right\Vert _{\mathcal{V}^{\varepsilon}}\le\frac{\kappa_{p}^{n}}{1-\kappa_{p}}\left\Vert u^{\varepsilon,1}\right\Vert _{\mathcal{V}^{\varepsilon}},
\]
which implies the convergence rate of the linearization and guarantees the stability of the problem $\left(P^{\varepsilon}\right)$.
\end{proof}

\section{Homogenization asymptotics. Corrector estimates}\label{HM}

\subsection{Two-scale asymptotic expansions}


For every $i\in\left\{ 1,...,N\right\} $, we introduce the following
$M$th-order expansion $\left(M\ge2\right)$:
\begin{equation}
u_{i}^{\varepsilon}\left(x\right)=\sum_{m=0}^{M}\varepsilon^{m}u_{i,m}\left(x,y\right)+\mathcal{O}\left(\varepsilon^{M+1}\right),\quad x\in\Omega^{\varepsilon},\label{eq:expansion}
\end{equation}
where $u_{i,m}\left(x,\cdot\right)$ is $Y$-periodic for $0\le m\le M$.

It follows from \eqref{eq:expansion} that
\begin{eqnarray}
\nabla u_{i}^{\varepsilon} & = & \left(\nabla_{x}+\varepsilon^{-1}\nabla_{y}\right)\left(\sum_{m=0}^{M}\varepsilon^{m}u_{i,m}+\mathcal{O}\left(\varepsilon^{M+1}\right)\right)\nonumber \\
 & = & \varepsilon^{-1}\nabla_{y}u_{i,0}+\sum_{m=0}^{M-1}\varepsilon^{m}\left(\nabla_{x}u_{i,m}+\nabla_{y}u_{i,m+1}\right)+\mathcal{O}\left(\varepsilon^{M}\right).\label{eq:laplaceu}
\end{eqnarray}

Using the relation of the operator $\mathcal{A}^{\varepsilon}$ and
\eqref{eq:laplaceu}, we compute that
\begin{eqnarray*}
\mathcal{A}_{\varepsilon}u_{i}^{\varepsilon} & = & \left(\nabla_{x}+\varepsilon^{-1}\nabla_{y}\right)\cdot\left(-d_{i}\left(y\right)\left[\varepsilon^{-1}\nabla_{y}u_{i,0}+\sum_{m=0}^{M-1}\varepsilon^{m}\left(\nabla_{x}u_{i,m}+\nabla_{y}u_{i,m+1}\right)\right]\right)\\
 &  & +\mathcal{O}\left(\varepsilon^{M-1}\right),
\end{eqnarray*}
then after collecting those having the same powers of $\varepsilon$,
we obtain
\begin{eqnarray}
\mathcal{A}_{\varepsilon}u_{i}^{\varepsilon} & = & \varepsilon^{-2}\nabla_{y}\cdot\left(-d_{i}\left(y\right)\nabla_{y}u_{i,0}\right)\nonumber \\
 &  & +\varepsilon^{-1}\left[\nabla_{x}\cdot\left(-d_{i}\left(y\right)\nabla_{y}u_{i,0}\right)+\nabla_{y}\cdot\left(-d_{i}\left(y\right)\left(\nabla_{x}u_{i,0}+\nabla_{y}u_{i,1}\right)\right)\right]\nonumber \\
 &  & +\sum_{m=0}^{M-2}\varepsilon^{m}\left[\nabla_{x}\cdot\left(-d_{i}\left(y\right)\left(\nabla_{x}u_{i,m}+\nabla_{y}u_{i,m+1}\right)\right)\right.\nonumber \\
 &  & \left.+\nabla_{y}\cdot\left(-d_{i}\left(y\right)\left(\nabla_{x}u_{i,m+1}+\nabla_{y}u_{i,m+2}\right)\right)\right]+\mathcal{O}\left(\varepsilon^{M-1}\right).\label{eq:newoperator}
\end{eqnarray}

In the same vein, we take into consideration the boundary condition
at $\Gamma^{\varepsilon}$ as follows:
\begin{eqnarray}
-d_{i}^{\varepsilon}\nabla u_{i}^{\varepsilon}\cdot\mbox{n} & := & -d_{i}\left(y\right)\left(\varepsilon^{-1}\nabla_{y}u_{i,0}+\sum_{m=0}^{M-1}\varepsilon^{m}\left(\nabla_{x}u_{i,m}+\nabla_{y}u_{i,m+1}\right)\right)\cdot\mbox{n}\nonumber \\
 & = & \varepsilon b_{i}\left(y\right)F_{i}\left(\sum_{m=0}^{M-1}\varepsilon^{m}u_{i,m}\right)-a_{i}\left(y\right)\sum_{m=0}^{M-1}\varepsilon^{m+1}u_{i,m}+\mathcal{O}\left(\varepsilon^{M}\right).\label{eq:boundary}
\end{eqnarray}

It is worth noting that in order to investigate the convergence analysis,
we give assumptions that allow to pull the $\varepsilon$-dependent
quantities out of the nonlinearities $R_{i}$ and $F_{i}$: 
\begin{equation}
R_{i}\left(\sum_{m=0}^{M}\varepsilon^{m}u_{1,m},...,\sum_{m=0}^{M}\varepsilon^{m}u_{N,m}\right)=\sum_{m=0}^{M}\varepsilon^{m}\bar{R}_{i}\left(u_{1,m},...,u_{N,m}\right)+\mathcal{O}\left(\varepsilon^{M+1}\right),\label{eq:newR}
\end{equation}
\begin{equation}
F_{i}\left(\sum_{m=0}^{M}\varepsilon^{m}u_{i,m}\right)=\sum_{m=0}^{M}\varepsilon^{m}\bar{F}_{i}\left(u_{i,m}\right)+\mathcal{O}\left(\varepsilon^{M+1}\right),\label{eq:newF}
\end{equation}
in which $\bar{R}_{i}$ and $\bar{F}_{i}$ are global Lipschitz functions
corresponding to the Lipschitz constant $L_{i}$ and $K_{i}$, respectively,
for $i\in\left\{ 1,...,N\right\} $.

From now on, collecting the coefficients of the same powers of $\varepsilon$
in \eqref{eq:newoperator} and \eqref{eq:boundary} in combination
with using \eqref{eq:newR} and \eqref{eq:newF}, we are led to the
following systems of elliptic problems, which we refer to the auxiliary
problems:
\begin{equation}
\begin{cases}
\mathcal{A}_{0}u_{i,0}=0, & \mbox{in}\;Y_{1},\\
-d_{i}\left(y\right)\nabla_{y}u_{i,0}\cdot\mbox{n}=0, & \mbox{on}\;\partial Y_{0},\\
u_{i,0}\;\mbox{is}\;Y-\mbox{periodic in}\;y,
\end{cases}\label{eq:prob1}
\end{equation}

\begin{equation}
\begin{cases}
\mathcal{A}_{0}u_{i,1}=-\mathcal{A}_{1}u_{i,0}, & \mbox{in}\;Y_{1},\\
-d_{i}\left(y\right)\left(\nabla_{x}u_{i,0}+\nabla_{y}u_{i,1}\right)\cdot\mbox{n}=0, & \mbox{on}\;\partial Y_{0},\\
u_{i,1}\;\mbox{is}\;Y-\mbox{periodic in}\;y,
\end{cases}\label{eq:prob2}
\end{equation}

\begin{equation}
\begin{cases}
\mathcal{A}_{0}u_{i,m+2}=\bar{R}_{i}\left(u_{m}\right)-\mathcal{A}_{1}u_{i,m+1}-\mathcal{A}_{2}u_{i,m}, & \mbox{in}\;Y_{1},\\
-d_{i}\left(y\right)\left(\nabla_{x}u_{i,m+1}+\nabla_{y}u_{i,m+2}\right)\cdot\mbox{n}=b_{i}\left(y\right)\bar{F}_{i}\left(u_{i,m}\right)-a_{i}\left(y\right)u_{i,m}, & \mbox{on}\;\partial Y_{0},\\
u_{i,m+2}\;\mbox{is}\;Y-\mbox{periodic in}\;y,
\end{cases}\label{eq:probm}
\end{equation}
for $0\le m\le M-2$.

Here, the notation $u_{m}$ is ascribed to the vector containing elements
$u_{i,m}$ for all $i\in\left\{ 1,...,N\right\} $, and we have denoted
by
\begin{eqnarray*}
\mathcal{A}_{0} & := & \nabla_{y}\cdot\left(-d_{i}\left(y\right)\nabla_{y}\right),\\
\mathcal{A}_{1} & := & \nabla_{x}\cdot\left(-d_{i}\left(y\right)\nabla_{y}\right)+\nabla_{y}\cdot\left(-d_{i}\left(y\right)\nabla_{x}\right),\\
\mathcal{A}_{2} & := & \nabla_{x}\cdot\left(-d_{i}\left(y\right)\nabla_{x}\right).
\end{eqnarray*}

For the first auxiliary problem \eqref{eq:prob1}, it is trivial to
prove that the solution to \eqref{eq:prob1} is independent of $y$,
and hence we obtain
\begin{equation}
u_{i,0}\left(x,y\right)=\tilde{u}_{i,0}\left(x\right).\label{eq:4.12}
\end{equation}

For the second auxiliary problem \eqref{eq:prob2}, we recall the
result in \cite[Lemma 2.1]{PPSW93} to ensure the existence and uniqueness
of periodic solutions to the elliptic problem, which is called the
solvability condition. In this case, this condition satisfies itself
because we easily get from the PDE in \eqref{eq:prob2} that
\[
-\int_{\partial Y_{1}}d_{i}\left(y\right)\nabla_{y}u_{i,1}\cdot\mbox{n}dS_{y}=\int_{\partial Y_{0}}d_{i}\left(y\right)\nabla_{x}\tilde{u}_{i,0}\cdot\mbox{n}dS_{y},
\]
by Gau\ss{}'s theorem. Thus, it claims the existence of a unique
weak solution to \eqref{eq:prob2}.

Moreover, this solution is sought by using separation of variables:
\begin{equation}
u_{i,1}\left(x,y\right)=-\chi_{i}\left(y\right)\cdot\nabla_{x}\tilde{u}_{i,0}\left(x\right)+\mathcal{C}_{i}\left(x\right).\label{eq:4.13}
\end{equation}

Substituting \eqref{eq:4.13} into \eqref{eq:prob2}, we obtain the
$i$th cell problem:
\begin{equation}
\begin{cases}
\mathcal{A}_{0}\chi_{i}=\nabla_{y}d_{i}\left(y\right), & \mbox{in}\;Y_{1},\\
-d_{i}\left(y\right)\nabla_{y}\chi_{i}\cdot\mbox{n}=d_{i}\left(y\right)\cdot\mbox{n}, & \mbox{on}\;\partial Y_{0},\\
\chi_{i}\;\mbox{is}\;Y-\mbox{periodic in}\;y,
\end{cases}\label{eq:xi}
\end{equation}
in which the field $\chi_{i}\left(y\right)$ is called 
cell function. Additionally, by the definition
of the mean value, we have
\begin{equation}
\mathcal{M}_{Y}\left(\chi_{i}\right):=\frac{1}{\left|Y\right|}\int_{Y_{1}}\chi_{i}dy=0.\label{eq:4.14}
\end{equation}

As a consequence, it can be proved that $\chi_{i}$ belongs to the
space $H_{\#}^{1}\left(Y_{1}\right)/\mathbb{R}$ and satisfies \eqref{eq:4.14}.

Now, it only remains to consider the third auxiliary problem \eqref{eq:probm}.
Assume that we have in mind the functions $u_{m}$ and $u_{m+1}$,
then to find $u_{m+2}$ let us remark that the right-hand side of
the PDE in \eqref{eq:probm} can be rewritten as
\begin{eqnarray}
\bar{R}_{i}\left(u_{m}\right)-\mathcal{A}_{1}u_{i,m+1}-\mathcal{A}_{2}u_{i,m} & = & \bar{R}_{i}\left(u_{m}\right)+\nabla_{y}\left(d_{i}\left(y\right)\nabla_{x}u_{i,m+1}\right)\nonumber \\
 &  & +\nabla_{x}\left(d_{i}\left(y\right)\left(\nabla_{x}u_{i,m}+\nabla_{y}u_{i,m+1}\right)\right).\label{eq:4.14-1}
\end{eqnarray}

We define the operator $\mathcal{L}_{i}\left(\psi\right)$ for $i\in\left\{ 1,...,N\right\} $
by multiplying \eqref{eq:4.14-1} by a test function $\psi\in C_{\#}^{\infty}\left(Y_{1}\right)$,
as follows:
\begin{eqnarray*}
\mathcal{L}_{i}\left(\psi\right) & = & \int_{Y_{1}}\bar{R}_{i}\left(u_{m}\right)\psi dy+\int_{Y_{1}}\nabla_{y}\left(d_{i}\left(y\right)\nabla_{x}u_{i,m+1}\right)\psi dy\\
 &  & +\int_{Y_{1}}\nabla_{x}\left(d_{i}\left(y\right)\left(\nabla_{x}u_{i,m}+\nabla_{y}u_{i,m+1}\right)\right)\psi dy\\
 &  & =\int_{Y_{1}}\bar{R}_{i}\left(u_{m}\right)\psi dy-\int_{Y_{1}}d_{i}\left(y\right)\nabla_{x}u_{i,m+1}\nabla_{y}\psi dy\\
 &  & +\int_{Y_{1}}\nabla_{x}\left(d_{i}\left(y\right)\left(\nabla_{x}u_{i,m}+\nabla_{y}u_{i,m+1}\right)\right)\psi dy.
\end{eqnarray*}

To apply the Lax-Milgram type lemma provided by \cite[Lemma 2.2]{CP99},
we need $\mathcal{L}_{i}\left(\psi_{1}\right)=\mathcal{L}_{i}\left(\psi_{2}\right)$
for $\psi_{1},\psi_{2}\in H_{\#}^{1}\left(Y_{1}\right)/\mathbb{R}$
with $\psi_{1}\simeq\psi_{2}$, or it is equivalent to
\begin{equation}
\int_{Y_{1}}\bar{R}_{i}\left(u_{m}\right)\left(\psi_{1}-\psi_{2}\right)dy+\int_{Y_{1}}\nabla_{x}\left(d_{i}\left(y\right)\left(\nabla_{x}u_{i,m}+\nabla_{y}u_{i,m+1}\right)\right)\left(\psi_{1}-\psi_{2}\right)dy=0.\label{eq:4.15}
\end{equation}

Note that $\psi_{1}-\psi_{2}$ is independent of $y$. Hence, \eqref{eq:4.15}
becomes
\begin{equation}
\int_{Y_{1}}\nabla_{x}\left(-d_{i}\left(y\right)\left(\nabla_{x}u_{i,m}+\nabla_{y}u_{i,m+1}\right)\right)dy=\int_{Y_{1}}\bar{R}_{i}\left(u_{m}\right)dy.\label{eq:13}
\end{equation}

For simplicity, we first take $m=0$. Remind from \eqref{eq:4.12}
and \eqref{eq:4.13} that $u_{i,0}$ and $u_{i,1}$ are known, while
the term $R_{i}\left(u_{0}\right)$ depends on $x$ only, then one
has
\[
\int_{Y_{1}}\nabla_{x}\left(-d_{i}\left(y\right)\left(-\nabla_{y}\chi_{i}\nabla_{x}\tilde{u}_{i,0}+\nabla_{x}\tilde{u}_{i,0}\right)\right)dy=\left|Y_{1}\right|\bar{R}_{i}\left(u_{0}\right),
\]
or equivalently,
\[
\int_{Y_{1}}\nabla_{x}\left(-d_{i}\left(y\right)\left(-\nabla_{y}\chi_{i}+\mathbb{I}\right)\nabla_{x}\tilde{u}_{i,0}\right)dy=\left|Y_{1}\right|\bar{R}_{i}\left(u_{0}\right).
\]

Thus, if we set the homogenized (or effective) coefficient 
\[
q_{i}=\frac{1}{\left|Y\right|}\int_{Y_{1}}d_{i}\left(y\right)\left(-\nabla_{y}\chi_{i}+\mathbb{I}\right)dy,
\]
the $\tilde{u}_{i,0}$ must satisfy (in the ``almost all''
sense)
\begin{equation}
-\nabla_{x}\left(q_{i}\nabla_{x}\tilde{u}_{i,0}\right)=\left|Y\right|^{-1}\left|Y_{1}\right|\bar{R}_{i}\left(u_{0}\right),\quad\mbox{in}\;\Omega.\label{eq:14}
\end{equation}

On the other hand, it is associated with $\tilde{u}_{i,0}=0$ at $\Gamma^{ext}$
and still satisfies the ellipticity condition.

Let us now determine $u_{i,2}$. At first, the PDE in \eqref{eq:probm}
(for $m=0$) is given by
\begin{equation}
\mathcal{A}_{0}u_{i,2}=\bar{R}_{i}\left(u_{0}\right)-d_{i}\left(y\right)\nabla_{y}\chi_{i}\nabla_{x}^{2}\tilde{u}_{i,0}-\nabla_{y}\left(d_{i}\left(y\right)\chi_{i}\right)\nabla_{x}^{2}\tilde{u}_{i,0}+d_{i}\left(y\right)\nabla_{x}^{2}\tilde{u}_{i,0},\quad\mbox{in}\;Y_{1}.\label{eq:13-1}
\end{equation}

Next, the boundary condition reads
\[
-d_{i}\left(y\right)\nabla_{y}u_{i,2}\cdot\mbox{n}=b_{i}\left(y\right)\bar{F}_{i}\left(u_{i,0}\right)-a_{i}\left(y\right)u_{i,0}-d_{i}\left(y\right)\chi_{i}\nabla_{x}^{2}\tilde{u}_{i,0}\cdot\mbox{n},\quad\mbox{on}\;\partial Y_{0}.
\]

Note that \eqref{eq:13-1} can be rewritten as
\[
\mathcal{A}_{0}u_{i,2}-\nabla_{y}\left(d_{i}\left(y\right)\chi_{i}\nabla_{x}^{2}\tilde{u}_{i,0}\right)=\bar{R}_{i}\left(u_{0}\right)-d_{i}\left(y\right)\left(\nabla_{y}\chi_{i}-\mathbb{I}\right)\nabla_{x}^{2}\tilde{u}_{i,0}.
\]

Using the relation \eqref{eq:14}, we have
\begin{equation}
\mathcal{A}_{0}u_{i,2}+\mathcal{A}_{0}\left(\chi_{i}\nabla_{x}^{2}\tilde{u}_{i,0}\right)=-\left|Y_{1}\right|^{-1}\left|Y\right|\nabla_{x}\left(q_{i}\nabla_{x}\tilde{u}_{i,0}\right)-d_{i}\left(y\right)\left(\nabla_{y}\chi_{i}-\mathbb{I}\right)\nabla_{x}^{2}\tilde{u}_{i,0}.\label{eq:16}
\end{equation}

Therefore, \eqref{eq:16} allows us to look for $u_{i,2}$ of the
form
\begin{equation}
u_{i,2}\left(x,y\right)=\theta_{i}\left(y\right)\nabla_{x}^{2}\tilde{u}_{i,0},\label{eq:u2}
\end{equation}
in which such a function $\theta_{i}$ is the solution of the following
problem
\begin{equation}
\begin{cases}
\mathcal{A}_{0}\left(\nabla_{y}\theta_{i}-\chi_{i}\right)=-\left|Y_{1}\right|^{-1}\left|Y\right|q_{i}-d_{i}\left(y\right)\left(\nabla_{y}\chi_{i}-\mathbb{I}\right), & \mbox{in}\;Y_{1},\\
-d_{i}\left(y\right)\left(\nabla_{y}\theta_{i}-\chi_{i}\right)\cdot\mbox{n}=b_{i}\left(y\right)\bar{F}_{i}\left(u_{i,0}\right)-a_{i}\left(y\right)u_{i,0}, & \mbox{on}\;\partial Y_{0},\\
\theta_{i}\;\mbox{is}\;Y-\mbox{periodic in}\;y.
\end{cases}\label{eq:18-1}
\end{equation}

In conclusion, we have derived an expansion of $u_{i}^{\varepsilon}\left(x\right)$
up to the second-order corrector. In particular, we deduced that
\begin{equation}
u_{i}^{\varepsilon}\left(x\right)=\tilde{u}_{i,0}\left(x\right)-\varepsilon\chi_{i}\left(\frac{x}{\varepsilon}\right)\cdot\nabla_{x}\tilde{u}_{i,0}\left(x\right)+\varepsilon^{2}\theta_{i}\left(\frac{x}{\varepsilon}\right)\nabla_{x}^{2}\tilde{u}_{i,0}\left(x\right)+\mathcal{O}\left(\varepsilon^{3}\right),\quad x\in\Omega^{\varepsilon},\label{eq:22}
\end{equation}
where $\tilde{u}_{i,0}$ can be solved by the microscopic problem
\eqref{eq:prob1}, $\chi_{i}$ satisfies the cell problem \eqref{eq:xi},
and $\theta_{i}$ satisfies the cell problem \eqref{eq:18-1}. Moreover,
the homogenized equation is defined in \eqref{eq:14}.

For the time being, it remains to derive the macroscopic equation
from the PDE for $u_{i,2}$ in \eqref{eq:probm} for $m=0$. When
doing so, the following solvability condition has to be fulfilled:
\begin{equation}
\int_{Y_{1}}\left(\bar{R}_{i}\left(u_{0}\right)-\mathcal{A}_{1}u_{i,1}-\mathcal{A}_{2}\tilde{u}_{i,0}\right)dy=\int_{\partial Y_{0}}\left(b_{i}\left(y\right)\bar{F}_{i}\left(\tilde{u}_{i,0}\right)-a_{i}\left(y\right)\tilde{u}_{i,0}+d_{i}\left(y\right)\nabla_{x}u_{i,1}\cdot\mbox{n}\right)dS_{y}.\label{eq:19}
\end{equation}

The left-hand side of \eqref{eq:19} can be rewritten as
\begin{equation}
\int_{Y_{1}}\bar{R}_{i}\left(u_{0}\right)dy+\int_{Y_{1}}\nabla_{y}\left(d_{i}\left(y\right)\nabla_{x}u_{i,1}\right)dy+\int_{Y_{1}}\nabla_{x}\left(d_{i}\left(y\right)\left(\nabla_{x}\tilde{u}_{i,0}+\nabla_{y}u_{i,1}\right)\right)dy.\label{eq:20}
\end{equation}

Let us consider the last two integrals in \eqref{eq:20}. In fact,
we have
\begin{equation}
\int_{Y_{1}}\nabla_{x}\left(d_{i}\left(y\right)\nabla_{x}\tilde{u}_{i,0}\right)dy=\nabla_{x}\cdot\left[\left(\int_{Y_{1}}d_{i}\left(y\right)dy\right)\nabla_{x}\tilde{u}_{i,0}\right]=\left(\int_{Y_{1}}d_{i}\left(y\right)dy\right):\nabla_{x}\nabla_{x}\tilde{u}_{i,0},\label{eq:21}
\end{equation}
where we have used the inner product (or exactly, double dot product)
between two matrices
\[
A:B:=\mbox{tr}\left(A^{T}B\right)=\sum_{ij}a_{ij}b_{ij}.
\]

In addition, by periodicity and Gau\ss{}'s theorem we obtain
\begin{equation}
\int_{Y_{1}}\nabla_{y}\left(d_{i}\left(y\right)\nabla_{x}u_{i,1}\right)dy=\int_{\partial Y_{0}}d_{i}\left(y\right)\nabla_{x}u_{i,1}\cdot\mbox{n}dS_{y}.
\end{equation}

Next, employing the double dot product again, we get
\begin{equation}
\int_{Y_{1}}\nabla_{x}\left(d_{i}\left(y\right)\nabla_{y}u_{i,1}\right)dy=-\int_{Y_{1}}\left(d_{i}\left(y\right)\nabla_{y}\chi_{i}\right)dy:\nabla_{x}\nabla_{x}\tilde{u}_{i,0}.\label{eq:23}
\end{equation}

Combining \eqref{eq:19}, \eqref{eq:21}-\eqref{eq:23} yields the
macroscopic equation:
\[
\left(\int_{Y_{1}}\left(d_{i}\left(y\right)-d_{i}\left(y\right)\nabla_{y}\chi_{i}\right)dy\right):\nabla_{x}\nabla_{x}\tilde{u}_{i,0}=\left\langle b_{i}\right\rangle \bar{F}_{i}\left(\tilde{u}_{i,0}\right)-\left\langle a_{i}\right\rangle \tilde{u}_{i,0}-\left|Y_{1}\right|\bar{R}_{i}\left(u_{0}\right),
\]
where we have denoted by
\begin{eqnarray*}
\left\langle a_{i}\right\rangle  & := & \int_{\partial Y_{0}}a_{i}\left(y\right)dy,\\
\left\langle b_{i}\right\rangle  & := & \int_{\partial Y_{0}}b_{i}\left(y\right)dy.
\end{eqnarray*}

Furthermore, this equation is associated with the boundary condition
$\tilde{u}_{i,0}=0$ at $\Gamma^{ext}$.

\subsection{Corrector estimates. Justification of the asymptotics}

From the point of view of applications, upper bound estimates on convergence
rates over the space $\mathcal{V}^{\varepsilon}$ in terms of quantitative
analysis tells how fast one can approximate both $u^{\varepsilon}$,
the solution of systems $\left(P^{\varepsilon}\right)$, and $\nabla u^{\varepsilon}$
by the asymptotic expansion \eqref{eq:22}. On the other hand, it
also gives rise to the question that: how much information on data
will we need via such averaging techniques?

We introduce the well-known cut-off function $m^{\varepsilon}\in C_{c}^{\infty}\left(\Omega\right)$
such that $\varepsilon\left|\nabla m^{\varepsilon}\right|\le C$ and
\[
m^{\varepsilon}(x):=\begin{cases}
1, & \mbox{if}\;\mbox{dist}\left(x,\Gamma\right)\le\varepsilon,\\
0, & \mbox{if}\;\mbox{dist}\left(x,\Gamma\right)\ge2\varepsilon.
\end{cases}
\]

For $i\in\left\{ 1,...,N\right\} $, we define the function $\Psi_{i}^{\varepsilon}$
by
\[
\Psi_{i}^{\varepsilon}:=\varphi_{i}^{\varepsilon}+\left(1-m^{\varepsilon}\right)\left(\varepsilon u_{i,1}+\varepsilon^{2}u_{i,2}\right),
\]
where we have denoted by
\[
\varphi_{i}^{\varepsilon}:=u_{i}^{\varepsilon}-\left(u_{i,0}+\varepsilon u_{i,1}+\varepsilon^{2}u_{i,2}\right).
\]

Due to the auxiliary problems \eqref{eq:prob1}-\eqref{eq:probm},
we have
\begin{equation}
\mathcal{A}^{\varepsilon}\varphi_{i}^{\varepsilon}=R_{i}\left(u^{\varepsilon}\right)-\bar{R}_{i}\left(u_{0}\right)-\varepsilon\left(\mathcal{A}_{2}u_{i,1}+\mathcal{A}_{1}u_{i,2}\right)-\varepsilon^{2}\mathcal{A}_{2}u_{i,2},\quad x\in\Omega^{\varepsilon},\label{eq:AA}
\end{equation}
while on the boundary $\Gamma^{\varepsilon}$, the function $\varphi_{i}^{\varepsilon}$
satisfies
\begin{equation}
-d_{i}^{\varepsilon}\nabla_{x}\varphi_{i}^{\varepsilon}\cdot\mbox{n}=\varepsilon^{2}d_{i}^{\varepsilon}\nabla_{x}u_{i,2}\cdot\mbox{n}+\varepsilon\left[a_{i}^{\varepsilon}\left(u_{i,0}-u_{i}^{\varepsilon}\right)+b_{i}^{\varepsilon}\left(F_{i}\left(u_{i}^{\varepsilon}\right)-\bar{F}_{i}\left(u_{i,0}\right)\right)\right].\label{eq:BB}
\end{equation}

Rewriting the above information, the function $\varphi_{i}^{\varepsilon}$
satisfies the following system:
\begin{equation}
\begin{cases}
\mathcal{A}_{\varepsilon}\varphi_{i}^{\varepsilon}=\bar{R}_{i}\left(u^{\varepsilon}\right)-\bar{R}_{i}\left(u_{0}\right)+\varepsilon g_{i}^{\varepsilon}, & \mbox{in}\;\Omega^{\varepsilon},\\
-d_{i}^{\varepsilon}\nabla_{x}\varphi_{i}^{\varepsilon}\cdot\mbox{n}=\varepsilon^{2}h_{i}^{\varepsilon}\cdot\mbox{n}+\varepsilon l_{i}^{\varepsilon}, & \mbox{on}\;\Gamma^{\varepsilon},
\end{cases}\label{eq:24}
\end{equation}
where we have denoted by
\begin{eqnarray*}
g_{i}^{\varepsilon} & := & -d_{i}\left(\frac{x}{\varepsilon}\right)\chi_{i}\left(\frac{x}{\varepsilon}\right)\nabla_{x}^{3}\tilde{u}_{i,0}+d_{i}\left(\frac{x}{\varepsilon}\right)\theta_{i}\left(\frac{x}{\varepsilon}\right)\nabla_{x}^{3}\tilde{u}_{i,0}\\
 &  & +\nabla_{y}\left(d_{i}\left(\frac{x}{\varepsilon}\right)\theta_{i}\left(\frac{x}{\varepsilon}\right)\right)\nabla_{x}^{3}\tilde{u}_{i,0}+\varepsilon d_{i}\left(\frac{x}{\varepsilon}\right)\theta_{i}\left(\frac{x}{\varepsilon}\right)\nabla_{x}^{4}\tilde{u}_{i,0},
\end{eqnarray*}
\[
h_{i}^{\varepsilon}:=d_{i}\left(\frac{x}{\varepsilon}\right)\theta_{i}\left(\frac{x}{\varepsilon}\right)\nabla_{x}^{3}\tilde{u}_{i,0},
\]
\[
l_{i}^{\varepsilon}:=a_{i}\left(\frac{x}{\varepsilon}\right)\left(\tilde{u}_{i,0}-u_{i}^{\varepsilon}\right)+b_{i}\left(\frac{x}{\varepsilon}\right)\left(F_{i}\left(u_{i}^{\varepsilon}\right)-\bar{F}_{i}\left(\tilde{u}_{i,0}\right)\right).
\]

Now, multiplying the PDE in \eqref{eq:24} by $\varphi_{i}\in V_{\varepsilon}$
for $i\in\left\{ 1,...,N\right\} $ and integrating by parts, we get
that
\begin{eqnarray}
\left\langle d_{i}^{\varepsilon}\varphi_{i}^{\varepsilon},\varphi_{i}\right\rangle _{V^{\varepsilon}} & = & \left\langle \bar{R}_{i}\left(u^{\varepsilon}\right)-\bar{R}_{i}\left(u_{0}\right),\varphi_{i}\right\rangle _{L^{2}\left(\Omega^{\varepsilon}\right)}+\varepsilon\left\langle g_{i}^{\varepsilon},\varphi_{i}\right\rangle _{L^{2}\left(\Omega^{\varepsilon}\right)}\nonumber \\
 &  & -\varepsilon\left\langle l_{i}^{\varepsilon},\varphi_{i}\right\rangle _{L^{2}\left(\Gamma^{\varepsilon}\right)}-\varepsilon^{2}\int_{\Gamma^{\varepsilon}}h_{i}^{\varepsilon}\cdot\mbox{n}\varphi_{i}dS_{\varepsilon}.\label{eq:4.29}
\end{eqnarray}
To guarantee all the derivatives appearing in $g_{i}^{\varepsilon}$
(up to higher order correctors), $\tilde{u}_{i,0}$, which is the
solution to \eqref{eq:14}, needs to be smooth enough, says $L^{\infty}\left(\Omega^{\varepsilon}\right)$
(cf. \cite{ADN59}), and the cell functions $\chi_{i}$ and $\theta_{i}$
to \eqref{eq:xi} and \eqref{eq:18-1}, respectively, belong at least
to $H_{\#}^{1}\left(Y_{1}\right)$ as derived above. Consequently,
it allows us to estimate $g_{i}^{\varepsilon}$ by an $\varepsilon$-independent
constant, i.e.
\begin{equation}
\left\Vert g_{i}^{\varepsilon}\right\Vert _{L^{2}\left(\Omega^{\varepsilon}\right)}\le C\quad\mbox{for all}\;i\in\left\{ 1,...,N\right\} .\label{eq:25}
\end{equation}

Furthermore, it is easily to estimate the integral including $h_{i}^{\varepsilon}$
in \eqref{eq:4.29} by the following (see, e.g. \cite{CP99}):
\[
\int_{\Gamma^{\varepsilon}}h_{i}^{\varepsilon}\cdot\mbox{n}dS_{\varepsilon}\approx C\varepsilon^{-1},
\]
which leads to
\begin{equation}
\left\Vert h_{i}^{\varepsilon}\cdot\mbox{n}\right\Vert _{L^{2}\left(\Gamma^{\varepsilon}\right)}\le C\varepsilon^{-1/2}.
\end{equation}

Now, it remains to estimate the third integral in \eqref{eq:4.29}.
Thanks to $\left(\mbox{A}_{2}\right)$ and \eqref{eq:newF}, we may
have
\begin{equation}
\left|\left\langle l_{i}^{\varepsilon},\varphi_{i}\right\rangle _{L^{2}\left(\Gamma^{\varepsilon}\right)}\right|\le C\left(1+\bar{K}_{i}\right)\left\Vert u_{i}^{\varepsilon}-\tilde{u}_{i,0}\right\Vert _{L^{2}\left(\Gamma^{\varepsilon}\right)}\left\Vert \varphi_{i}\right\Vert _{L^{2}\left(\Gamma^{\varepsilon}\right)}.
\end{equation}

In the same vein, we get:
\begin{equation}
\left|\left\langle R_{i}\left(u^{\varepsilon}\right)-\bar{R}_{i}\left(u_{0}\right),\varphi_{i}\right\rangle _{L^{2}\left(\Omega^{\varepsilon}\right)}\right|\le C\bar{L}_{i}\left\Vert u^{\varepsilon}-\tilde{u}_{0}\right\Vert _{\mathcal{V}^{\varepsilon}}\left\Vert \varphi_{i}\right\Vert _{L^{2}\left(\Omega^{\varepsilon}\right)}.\label{eq:4.34}
\end{equation}

Combining \eqref{eq:4.29}-\eqref{eq:4.34} with $\left(\mbox{A}_{1}\right)$
and putting $\bar{L}:=\max\left\{ \bar{L}_{1},...,\bar{L}_{N}\right\} $
and $\bar{K}:=1+\max\left\{ \bar{K}_{1},...,\bar{K}_{N}\right\} $,
we are led to the estimate:
\begin{eqnarray}
\alpha\sum_{i=1}^{N}\left|\left\langle \varphi_{i}^{\varepsilon},\varphi_{i}\right\rangle _{V^{\varepsilon}}\right| & \le & C\left(\bar{L}\left\Vert u^{\varepsilon}-\tilde{u}_{0}\right\Vert _{\mathcal{V}^{\varepsilon}}+\varepsilon\right)\left\Vert \varphi\right\Vert _{\mathcal{V}^{\varepsilon}}+C\left(\bar{K}\left\Vert u^{\varepsilon}-\tilde{u}_{0}\right\Vert _{\mathcal{H}\left(\Gamma^{\varepsilon}\right)}+\varepsilon^{3/2}\right)\left\Vert \varphi\right\Vert _{\mathcal{H}\left(\Gamma^{\varepsilon}\right)}\nonumber \\
 & \le & C\left(\varepsilon+\varepsilon^{1/2}\right)\left\Vert \varphi\right\Vert _{\mathcal{V}^{\varepsilon}}\le C\varepsilon^{1/2}\left\Vert \varphi\right\Vert _{\mathcal{V}^{\varepsilon}},\label{eq:4.35-1}
\end{eqnarray}
where we have made use of the trace inequality $\left\Vert \varphi\right\Vert _{\mathcal{H}\left(\Gamma^{\varepsilon}\right)}\le C\varepsilon^{-1/2}\left\Vert \varphi\right\Vert _{\mathcal{V}^{\varepsilon}}$
and the Poincar\'e inequality $\left\Vert \varphi\right\Vert _{\mathcal{H}\left(\Omega^{\varepsilon}\right)}\le C\left\Vert \varphi\right\Vert _{\mathcal{V}^{\varepsilon}}$.

Recall  that our aim is to estimate  $\left\Vert \Psi^{\varepsilon}\right\Vert _{\mathcal{V}^{\varepsilon}}$,
it remains to control the term $\left\langle \left(1-m^{\varepsilon}\right)\left(\varepsilon u_{i,1}+\varepsilon^{2}u_{i,2}\right),\varphi_{i}\right\rangle _{V^{\varepsilon}}$
for $\varphi_{i}\in V^{\varepsilon}$. In fact, one easily has that
\begin{eqnarray}
\sum_{i=1}^{N}\left|\left\langle \left(1-m^{\varepsilon}\right)\left(\varepsilon u_{i,1}+\varepsilon^{2}u_{i,2}\right),\varphi_{i}\right\rangle _{V^{\varepsilon}}\right| & \le & C\varepsilon\left\Vert \nabla\left(1-m^{\varepsilon}\right)\right\Vert _{\mathcal{H}\left(\Omega^{\varepsilon}\right)}\left\Vert \varphi\right\Vert _{\mathcal{V}^{\varepsilon}}\nonumber \\
 &  & +C\left\Vert \left(1-m^{\varepsilon}\right)\nabla\left(\varepsilon u_{1}+\varepsilon^{2}u_{2}\right)\right\Vert _{\mathcal{H}\left(\Omega^{\varepsilon}\right)}\left\Vert \varphi\right\Vert _{\mathcal{V}^{\varepsilon}}\nonumber \\
 & \le & C\left(\varepsilon^{1/2}+\varepsilon^{3/2}\right)\left\Vert \varphi\right\Vert _{\mathcal{V}^{\varepsilon}}\le C\varepsilon^{1/2}\left\Vert \varphi\right\Vert _{\mathcal{V}^{\varepsilon}},\label{eq:4.35}
\end{eqnarray}
where we have used
\[
\left\Vert \nabla\left(1-m^{\varepsilon}\right)\right\Vert _{\mathcal{H}\left(\Omega^{\varepsilon}\right)}^{2}\le N\left(\int_{\Omega^{\varepsilon}\cap\left\{ x|\mbox{dist}\left(x,\Gamma\right)\le2\varepsilon\right\} }\left|\nabla m^{\varepsilon}\right|^{2}dx\right)\le C\varepsilon^{-1},
\]

\begin{eqnarray*}
\left\Vert \left(1-m^{\varepsilon}\right)\nabla\left(\varepsilon u_{1}+\varepsilon^{2}u_{2}\right)\right\Vert _{\mathcal{H}\left(\Omega^{\varepsilon}\right)}^{2} & \le & N\varepsilon^{2}\left|\Omega^{\varepsilon}\right|\int_{\Omega^{\varepsilon}\cap\left\{ x|\mbox{dist}\left(x,\Gamma\right)\le2\varepsilon\right\} }\left|\nabla m^{\varepsilon}\right|^{2}dx\\
 & \le & C\varepsilon^{3}.
\end{eqnarray*}

Hence, by using the triangle inequality in \eqref{eq:4.35-1} and
\eqref{eq:4.35} yields that
\[
\sum_{i=1}^{N}\left|\left\langle \Psi_{i}^{\varepsilon},\varphi_{i}\right\rangle _{V^{\varepsilon}}\right|\le C\varepsilon^{1/2}\left\Vert \varphi\right\Vert _{\mathcal{V}^{\varepsilon}},
\]
which finally leads to
\[
\left\Vert \Psi^{\varepsilon}\right\Vert _{\mathcal{V}^{\varepsilon}}\le C\varepsilon^{1/2},
\]
by choosing $\varphi=\Psi^{\varepsilon}$.

Summarizing, we can now state 
of the following theorem.
\begin{thm}
\label{thm:10}Let $u^{\varepsilon}$ be the solution of the elliptic
system $\left(P^{\varepsilon}\right)$ with assumptions $\left(\mbox{A}_{1}\right)-\left(\mbox{A}_{3}\right)$
and \eqref{eq:newR}-\eqref{eq:newF} up to $M=2$. Suppose that the
unique pair $\left(u_{0},u_{m}\right)\in\mathcal{W}^{\infty}\left(\Omega^{\varepsilon}\right)\times\mathcal{W}^{\infty}\left(\Omega^{\varepsilon};H_{\#}^{1}\left(Y_{1}\right)/\mathbb{R}\right)$
for $m\in\left\{ 1,2\right\} $. The following corrector with second
order for the homogenization limit holds:
\[
\left\Vert u^{\varepsilon}-u_{0}-m^{\varepsilon}\left(\varepsilon u_{1}+\varepsilon^{2}u_{2}\right)\right\Vert _{\mathcal{V}^{\varepsilon}}\le C\varepsilon^{1/2},
\]
where $u_{0},u_{1}$ and $u_{2}$ are vectors whose elements are defined
by \eqref{eq:4.12}, \eqref{eq:4.13} and \eqref{eq:u2}, respectively.\end{thm}

\section{Discussion}\label{D}

In real-world applications, the nonlinear reaction term $R_{i}$ is often locally Lipschitz. However, relying on  Lemma \ref{lem:3} the $L^{\infty}$-type estimate of the positive
solution makes the nonlinearity globally Lipschitz.  For example, we choose 
$N=2$ and only consider the $R_{1}\left(u_{1},u_{2}\right)=u_{1}u_{2}-u_{1}^{2}$.
We have
\[
\left|R_{1}\left(u_{1},u_{2}\right)-R_{1}\left(v_{1},v_{2}\right)\right|\le\max\left\{ \left\Vert u_{2}\right\Vert _{L^{\infty}},\left\Vert u_{1}\right\Vert _{L^{\infty}}+\left\Vert v_{1}\right\Vert _{L^{\infty}}\right\} \left(\left|u_{1}-v_{1}\right|+\left|u_{2}-v_{2}\right|\right).
\]

In addition, for $M=1$ we compute that
\begin{equation}
R_{1}\left(u_{1,0}+\varepsilon u_{1,1},u_{2,0}+\varepsilon u_{2,1}\right)=u_{1,0}u_{2,0}+\varepsilon\left(u_{1,1}u_{2,0}+u_{1,0}u_{2,1}-2u_{1,0}u_{1,1}\right)+\mathcal{O}\left(\varepsilon^{2}\right).\label{eq:4.38}
\end{equation}

Consequently, it follows from \eqref{eq:4.38} that
\begin{eqnarray*}
R_{1}\left(\sum_{m\in\left\{ 0,1\right\} }\varepsilon^{m}u_{1,m},\sum_{m\in\left\{ 0,1\right\} }\varepsilon^{m}u_{2,m}\right) & = & \sum_{m\in\left\{ 0,1\right\} }\varepsilon^{m}\left[\left(1-m\right)u_{1,0}u_{2,0}+\right.,\\
 &  & \left.+m\left(u_{1,1}u_{2,0}+u_{1,0}u_{2,1}-2u_{1,0}u_{1,1}\right)\right]+\mathcal{O}\left(\varepsilon^{2}\right).
\end{eqnarray*}
which implies $\bar{R}_{1}:=\left(1-m\right)u_{1,0}u_{2,0}+m\left(u_{1,1}u_{2,0}+u_{1,0}u_{2,1}-2u_{1,0}u_{1,1}\right)$.

If $u_{i,m},v_{i,m}\in L^{\infty}\left(\Omega^{\varepsilon}\right)$
for all $i,m$ we thus arrive at
\[
\left|\bar{R}_{1}\left(u_{1,0},u_{1,1},u_{2,0},u_{2,1}\right)-\bar{R}_{1}\left(v_{1,0},v_{1,1},v_{2,0},v_{2,1}\right)\right|\le L_{1}\sum_{m\in\left\{ 0,1\right\} ,i\in\left\{ 1,2\right\} }\left|u_{i,m}-v_{i,m}\right|,
\]
where $L_{1}=4\max\left\{ \left\Vert u_{2,0}\right\Vert _{L^{\infty}\left(\Omega^{\varepsilon}\right)},\left\Vert v_{1,0}\right\Vert _{L^{\infty}\left(\Omega^{\varepsilon}\right)},\left\Vert v_{1,1}\right\Vert _{L^{\infty}\left(\Omega^{\varepsilon}\right)},\left\Vert v_{2,1}\right\Vert _{L^{\infty}\left(\Omega^{\varepsilon}\right)},\left\Vert u_{1,0}\right\Vert _{L^{\infty}\left(\Omega^{\varepsilon}\right)},1\right\} $.

A similar discussion for the nonlinear surface rates $F_i$.  In particular, note that that if $L^\infty$ bounds are available (up to the boundary) then also the  exponential function $F\left(u\right)=e^{u}$ can be treated conveniently. 


We may repeat the homogenization procedure by the auxiliary problems
\eqref{eq:prob1}-\eqref{eq:probm} to obtain not only the general
expansion of the concentrations and corresponding problems, but also the higher order of corrector estimate due to
the $\tilde{u}_{0}$-based construction of $u_{m}$. Taking the $M$-level
expansion \eqref{eq:expansion} into consideration, the general corrector
can be found easily. Indeed, by induction we have from \eqref{eq:AA}
that for $x\in\Omega^{\varepsilon}$ 
\begin{eqnarray*}
\mathcal{A}^{\varepsilon}\varphi_{i}^{\varepsilon} & = & \mathcal{A}^{\varepsilon}u_{i}^{\varepsilon}-\varepsilon^{-2}\mathcal{A}_{0}u_{i,0}-\varepsilon^{-1}\left(\mathcal{A}_{0}u_{i,1}+\mathcal{A}_{1}u_{i,0}\right)\\
 &  & -\sum_{m=0}^{M-2}\varepsilon^{m}\left(\mathcal{A}_{0}u_{i,m+2}+\mathcal{A}_{1}u_{i,m+1}+\mathcal{A}_{2}u_{i,m}\right)\\
 &  & -\varepsilon^{M-1}\left(\mathcal{A}_{1}u_{i,M}+\mathcal{A}_{2}u_{i,M-1}\right)-\varepsilon^{M}\mathcal{A}_{2}u_{i,M}\\
 & = & R_{i}\left(u^{\varepsilon}\right)-\sum_{m=0}^{M-2}\varepsilon^{m}\bar{R}_{i}\left(u_{m}\right)-\varepsilon^{M-1}\left(\mathcal{A}_{1}u_{i,M}+\mathcal{A}_{2}u_{i,M-1}\right)-\varepsilon^{M}\mathcal{A}_{2}u_{i,M},
\end{eqnarray*}
while \eqref{eq:BB} becomes
\[
-d_{i}^{\varepsilon}\nabla_{x}\varphi_{i}^{\varepsilon}\cdot\mbox{n}=\varepsilon^{M}d_{i}^{\varepsilon}\nabla_{x}u_{i,M}+\varepsilon\left[a_{i}^{\varepsilon}\left(\sum_{m=0}^{M-2}\varepsilon^{m}u_{i,m}-u_{i}^{\varepsilon}\right)+b_{i}^{\varepsilon}\left(F\left(u_{i}^{\varepsilon}\right)-\sum_{m=0}^{M-2}\varepsilon^{m}\bar{F}\left(u_{i,m}\right)\right)\right].
\]

Thanks to the assumptions \eqref{eq:newR} and \eqref{eq:newF}, we
are totally in a position to prove the generalization of Theorem \ref{thm:10}.
Since we just need to follow the above procedure, we shall give the
following theorem while skipping the proof.

\begin{thm}
\label{thm:13}Let $u^{\varepsilon}$ be the solution of the elliptic
system $\left(P^{\varepsilon}\right)$ with assumptions $\left(\mbox{A}_{1}\right)-\left(\mbox{A}_{3}\right)$
and \eqref{eq:newR}-\eqref{eq:newF} up to $M$-level of expansion.
Suppose that the unique pair $\left(u_{0},u_{m}\right)\in\mathcal{W}^{\infty}\left(\Omega^{\varepsilon}\right)\times\mathcal{W}^{\infty}\left(\Omega^{\varepsilon};H_{\#}^{1}\left(Y_{1}\right)/\mathbb{R}\right)$
for all $0\le m\le M$. The following correctors for the homogenization
limit hold:
\[
\left\Vert u^{\varepsilon}-\sum_{m=0}^{M}\varepsilon^{m}u_{m}\right\Vert _{\mathcal{V}^{\varepsilon}}\le C\left(\varepsilon^{M-1}+\varepsilon^{M-1/2}\right),
\]
\[
\left\Vert u^{\varepsilon}-u_{0}-m^{\varepsilon}\sum_{m=1}^{M}\varepsilon^{m}u_{m}\right\Vert _{\mathcal{V}^{\varepsilon}}\le C\sum_{m=1}^{M}\varepsilon^{m-1/2}.
\]

\end{thm}

\section*{Acknowledgment}
AM thanks NWO MPE "Theoretical estimates of heat losses in geothermal wells" (grant nr. 657.014.004) for funding. VAK gratefully acknowledges the hospitality of Department of Mathematics and Computer Science, Eindhoven University of Technology in The Netherlands.

\section*{References}

\bibliography{mybib}

\end{document}